\documentclass{amsart}

\usepackage{amssymb}
\usepackage{graphicx}
\usepackage[cmtip,all]{xy}

\newcommand{\Cijk}{CH^{i+j-k}(\Ymm{2k-i+1})}
\newcommand{\cijk}{C_j^{i,k}}
\newcommand{\Kbar}{\bar{K}}
\newcommand{\cij}{C_j^{i}}
\newcommand{\one}{\boldsymbol{1}}
\newcommand{\alb}{\operatorname{Alb}}
\newcommand{\dlp}{\frac{d}{d \log (p)}}
\newcommand{\XX}{\mathcal{X}}
\renewcommand{\O}{\mathcal{O}}
\newcommand{\Xgrph}{\Gamma(X)}
\newcommand{\red}{\operatorname{red}}

\newcommand{\rf}{\kappa}
\renewcommand{\O}{\mathcal{O}}
\newcommand{\Un}{\operatorname{Un}}
\newcommand{\ffm}{M_{\varphi,N,F^{\bullet}}(K)}
\newcommand{\myi}{\iota}
\newcommand{\mysplit}{\text{\ss}}
\newcommand{\gm}{\mathbb{G}_m}
\newcommand{\Ker}{\operatorname{Ker}}
\newcommand{\pair}[1]{\langle #1 \rangle}
\newcommand{\nill}{\operatorname{Nil}}
\newcommand{\myp}{\rho}
\newcommand{\DR}{\operatorname{DR}}
\newcommand{\res}{\operatorname{res}}
\newcommand{\ord}{\operatorname{ord}}
\newcommand{\Dst}{\operatorname{D}_{\textup{st}}}
\newcommand{\cst}{C_{\textup{st}}^\bullet}
\newcommand{\csp}{C_{\textup{st}}^{\prime\bullet}}

\newcommand{\bQ}{{\mathbb{Q}}}
\newcommand{\Q}{{\mathbb{Q}}}
\newcommand{\Z}{{\mathbb{Z}}}
\newcommand{\Zl}{{\mathbb{Z}_\ell}}
\newcommand{\Ql}{{\mathbb{Q}_\ell}}
\newcommand{\Qp}{{\bQ_p}}

\newcommand{\hst}{H_{\textup{st}}}

\newcommand{\hm}{H_{\mathcal{M}}}
\newcommand{\hdr}{H_{\textup{dR}}}
\newcommand{\hsyn}{H_{\textup{syn}}}
\newcommand{\rsyn}{\textup{reg}_{\textup{syn}}}
\newcommand{\reg}{\textup{reg}}
\newcommand{\regt}{\reg_t}
\newcommand{\isom}{\cong}
\newcommand{\acol}{\mathcal{O}_{\textup{Col}}}
\newcommand{\val}{\operatorname{val}}
\renewcommand{\L}{\mathcal{L}}

\newcommand{\coker}{\operatorname{CoKer}}
\newcommand{\nt}{\tilde{N}}
\newcommand{\hh}{\mathcal{H}}
\newcommand{\et}{\text{\'et}} \newcommand{\het}{H_{\et}}
\newcommand{\ntl}{\nt_\ell}
\newcommand{\gr}{\operatorname{gr}}
\newcommand{\Gal}{\operatorname{Gal}}
\newcommand{\Ymm}[1]{\bar{Y}^{(#1)}}
\newcommand{\Ym}{\Ymm{m}}

\def\hT(#1,#2,#3){H_{\mathcal{T}}^{#2}(#1,\Z(#3))}
\def\hmot(#1,#2,#3){H_{\mathcal{M}}^{#2}(#1,\Z(#3))}

\newtheorem{conjecture}{Conjecture}

\newtheorem{theorem}{Theorem}[section]
\newtheorem{lemma}[theorem]{Lemma}
\newtheorem{corollary}[theorem]{Corollary}
\newtheorem{proposition}[theorem]{Proposition}

\theoremstyle{definition}
\newtheorem{definition}[theorem]{Definition}

\theoremstyle{remark}
\newtheorem{remark}[theorem]{Remark}

\numberwithin{equation}{section}

\begin{document}

\title[Derivatives of Vologodsky functions]{Regulators and derivatives of Vologodsky functions with respect to
$\log(p)$}
\date{\today}
\author{Amnon Besser}
\address{
Ben-Gurion University of the Negev}
\email{bessera@math.bgu.ac.il}
\thanks{}
\subjclass[2020]{Primary }

\date{}

\begin{abstract}
  We describe several instances of the following phenomenon: In bad
  reduction situations the \( p \)-adic regulator has a continuous and a
  discrete component. The continuous component is computed using Vologodsky
  integrals. These depend on a choice of the branch of the \( p \)-adic
  logarithm, determined by \( \log (p) \). They can be differentiated with
  respect to this parameter and the result is related to the discrete
  component.
\end{abstract}

\maketitle


\section{Introduction}
Let \( p \) be a finite prime. Let $K$ be a finite extension of the field $\Q_p$ of \( p \)-adic numbers and let 
$X/ K$ be a smooth variety.

The \( p \)-adic analogue of the Beilinson regulator into Deligne cohomology is the \emph{syntomic regulator} 
\begin{equation*}
  \rsyn: \hm^i(X,\Q(j)) \to \hsyn^i(X,j)
\end{equation*} 
\cite{Nek-Niz14}.

The syntomic regulator $\rsyn$ can sometimes be computed using Coleman or Vologodsky integration.
Let us consider a concrete example of the second algebraic K-theory group
$K_2(X)$ of a proper curve $X$. In this case the regulator takes the form
\begin{equation}\label{k2reg}
  \rsyn: \hm^2(X,\Q(2))\to
\hsyn^2(X,2)\;.
\end{equation}
Elements of the left-hand side of~\eqref{k2reg} are represented by formal sums of symbols
\( \{f,g\} \) for rational functions \( f,g \) on \( X \), such that all
their tame symbols are \( 0 \) at all points of \( X \).
When \( X \) has good reduction the group \( \hsyn^2(X,2) \) is isomorphic to the first algebraic de Rham cohomology group \( \hdr^1(X / K) \). In this case we proved
in~\cite{Bes98b} the following result.

\begin{theorem}\label{ktwogood}
To $\omega\in \hdr^1(X / K)$ associate $r_\omega = \omega \cup \rsyn$. Then,
for $\omega\in \Omega^1(X)$,
  \begin{equation}\label{regk2}
r_\omega(\{f,g\})= \int_{(f)} \log(g) \omega\;.
\end{equation}
\end{theorem}
The expression on the right-hand side of~\eqref{regk2} is the Coleman
integral (see Section~\ref{sec:colvol}) of the form \( \log (g)\cdot \Omega \) evaluated on the divisor \(
(f) \) of the rational function \( f \).
We note that this theorem is a special case of a more general formula for a general \( \omega\in \hdr^1(X /K) \).

When \( X \) only has semi-stable reduction over \( K \), we have an isomorphism
\begin{equation}\label{totosyn}
  \hsyn^2(X,2)\isom \hdr^1(X /K)\oplus B\;,
\end{equation} 
where \( B \) is a subset of the weight \( 2 \) part of \( \hdr^1(X / K) \)
(weight with respect to Frobenius, see Section~\ref{sec:syntomic}). Focusing for the moment on the first summand
in~\eqref{totosyn}
(though the second component will be our main focus) we proved
in~\cite{Bes18} an almost identical result by replacing Coleman integration with Vologodsky integration.
\begin{theorem}[{\cite[Corollary~1.3]{Bes18}}]\label{k2ssthm}
    Under some technical assumptions on the form \( \omega \), the first
    component of the syntomic regulator is computed by~\eqref{regk2},
    replacing the Coleman integral 
    with a Vologodsky integrals
\end{theorem}
See Theorem~\ref{k2sstwo} for a more precise statement.

Vologodsky~\cite{Vol01} integration extend Coleman integration to varieties
with bad reduction. Vologodsky integrals have, built into them, a branch of
the \( p \)-adic logarithm, determined by the branch parameter \( \log (p)
\). In fact, in Theorem~\ref{k2ssthm} the regulator also depends on
(essentially) the branch parameter.

One can wonder about the dependency of the Vologodsky integral on
the branch of the logarithm. We may even take the derivative of the
integral
with respect to the branch parameter \( \log (p) \). We can ask
\emph{what can we say about this derivative}. The goal of this note it to
argue that this is an interesting question since the resulting derivative
holds valuable arithmetic information.

The fundamental example of this phenomenon is the $p$-adic logarithm \(
\log (z) \), which is the Vologodsky integral of the differential form \(
dz / z \) on the punctured affine line.
Recall that the \( p \)-adic logarithm (sometimes called the Iwasawa
logarithm) of \( z\in K \) with \( |z|=1 \) is uniquely determined by
requiring that it has the usual power series expansion near \( 1 \) and
that it satisfies \( \log (ab)=\log (a)+\log (b) \). However, to extend it
to a homomorphism \( \log : K^\times \to K \) one needs to make a choice of
\( \log (p) \). Once such a choice has been made the extension of the
logarithm to \( K^\times  \) is given as follows:
If $\nu$ is the $p$-adic valuation normalized in such a way that $\nu(p)=1$, then
$$\log(z) = \log\left(\frac{z}{p^{\nu(z)}}\right) + \nu(z) \log(p)\;.$$
By differentiating with respect to \( \log (p) \) we get the fundamental
observation of this entire note:
\begin{equation}\label{key-thing}
\frac{d}{d
\log(p)} \log(z) = \nu(z)\;.
\end{equation} 
To see how this observation relates to the theory of \( p \)-adic
regulators, let us consider the simplest such (log syntomic) regulator, namely the one for
\( K_1 \) of a field \( K \), \( \hm^1(K,1) \). 
This takes the form of the map
\[
   K^{\times }\to \hst^1(K,\Q_p(1))=K\oplus \Q_p
\]
where \( \hst \) is semi-stable Galois cohomology~\cite{Nek93}.
By~\cite[1.35]{Nek93}
this regulator is given by
\[
x\mapsto (\log(x),\nu(x))\;.
\] 

We immediately observe the main theme of this note:
The regulator has 2 components, one ``continuous'' and one ``discrete''. The
former is computed using Vologodsky integration, The latter is the
derivative of the former with respect to $\log(p)$.
We will try to convince the reader that the phenomena that we point out is not
only interesting on its own, but is potentially useful as a tool, mostly
for computing discrete invariants using continuous tools.

To our knowledge, the first place this idea appears in any form is the work
of Colmez~\cite[Th\'eor\`eme~II.2.18]{Colm96}, where the local height at a
prime is found as the derivative with respect to \( \log (p) \) of the
appropriate Green function. We will prove something similar in
Proposition~\ref{ascolmez}. See 
Remark~\ref{colmremark} for the difference between the two approaches.

We will survey a number of examples where a similar picture emerges. In
Section~\ref{sec:colvol} we recall in brief the theories of Coleman and
Vologodsky, including its dependency on \( \log (p) \). Then, in
Section~\ref{sec:syntomic} We show, under fairly general conditions, that
the syntomic regulator for varieties with semi-stable reduction splits into
a continuous and discrete components, and that the discrete one is the
derivative of the continuous one with respect to \( \log (p) \). In
Section~\ref{sec:deriv} we explain how to compute the derivative with
respect to \( \log (p) \) of a Vologodsky integral in some cases, and give
an application to local height pairings on curves. Section~\ref{ss-loch} is
mostly based on~\cite{BMS21} and we explain the theory of \( p \)-adic
heights associated to adelic line bundles and how the local heights can be
computed using derivatives of Vologodsky functions. In
Section~\ref{sec:unipalb} we do a non-abelian analogue of
Section~\ref{sec:syntomic}, compute the continuous component of the
unioptent Albanese map using  Vologodsky integrals and the discrete part
using their derivatives with respect to \( \log (p) \). Finally, in
Section~\ref{sec:toric} we explain the relation with the theory of the
toric regulator developed in~\cite{Bes-Ras17}.

The presentation is not meant to be complete, and we will sometimes refer to
future work for details.

Notation: \( K \) is a finite extension of \( \mathbb{Q}_p \), unless
stated otherwise. In this case, its residue field \( \rf \) is finite and has \(
q=p^f \) elements. We denote by $K_0$ the maximal unramified extension of
$\Qp$ inside $K$.

I would like to thank the organizers of the Regulators V conference in
Pisa, who gave me the opportunity to present this material, both at the
conference and in writing. I would also like to thank Steffen M\"uller and
Padma Srinivasan, my coauthors on~\cite{BMS21,BMS23}, where the inspiration
for the ideas presented here grew, and for allowing me to report on some
results that will appear in~\cite{BMS23}. The author's research is
supported by grant number 2958/24 from the Israel Science foundation.

\section{Review of Coleman and Vologodsky integration}\label{sec:colvol}

Coleman integration is a theory of \( p \)-adic integration invented by
Robert Coleman and developed by him in~\cite{Col82,Col85,Col-de88}. It is
an integration theory (what of? A bit later) on certain ``overconvergent'' domains with good
reduction. Coleman called these: \emph{basic wide open spaces}, but a consistent
theory of Dagger spaces, associated to the Dagger algebras of Monsky and
Washnitzer~\cite{Mon-Was68}, was later developed by Grosse-Kl{\"o}nne~\cite{Grosse00}
and is the right ``home'' for this theory. If
\( X \) is a smooth variety over \( K \), then the associated rigid
analytic space can often be covered by such domains. These domains can be
chosen with good reduction even if \( X \) has bad reduction, and therefore
one can sometimes glue integrals over different domains. In general, there
will be some monodromy for the covering (see~\cite{Ber07} for a very
general theory along those lines). The Tannakian interpretation of Coleman
integration~\cite{Bes99} allows, among other things, an extension of the
theory to more general domains without the need for covering, in particular
to complete varieties, but still
assuming good reduction.

In the bad reduction case covering by overconvergent domains, if
they exist, will in general produce some monodromy. Nevertheless,
Vologodsky~\cite{Vol01} defines an integration theory for general varieties
with arbitrary reduction.

The integrals appearing in the theories of Coleman and Vologodsky are of
the following form: Let \( X \) be either an ``overconvergent domain'' for
Coleman integration or a smooth \( K \) variety for Vologodsky integration.
Let \( \omega \) be a closed one-form on \( X \). These would be
overconvergent analytic for Coleman or algebraic for Vologodsky. Then the
theory provides a primitive, A ``function'' \( F = \int \omega  \) on \( X \) in an
appropriate sense, so that \( dF=\omega \). The function \( F \) is a
locally analytic function on the \( \overline{K} \)-points, so that the
equation \( dF = \omega \) makes sense, and it is Galois equivariant. It is
a function on \( X \) in the Vologodsky case but is only a function on the
``underlying affinoid'' in Coleman's terminology (or the associated rigid spaces in Grosse-Kl{\"o}nne's
terminology).  Now, we can further multiply \( F \) by a second form \( \eta
\) and the theory gives an integral \( \int (F \eta) \), again unique up to
a constant, and so on (to be more precise, we can only integrate a
combination of such expressions which is again closed. The underlying
theory was developed in~\cite{Bes99} as Coleman originally only considered
one dimensional spaces, where this issue does not arise).
More concretely,
the integration
theory 
gives a $K$-algebra $\acol(X)\supset \O(X)$ of locally analytic
functions on $X$ such that there is a short exact sequence
\begin{equation*}
    0\to K \to \acol(X) \xrightarrow{d} \acol(X)\cdot \Omega^1(X)^{d=0}\to
    0
\end{equation*} 
and everything is functorial with respect to morphisms of ``spaces''.
All earlier described integrals are easily obtained as a consequence of
this short exact sequence. Definite, or more generally \emph{iterated}
integrals, are obtained from this theory just as in first year calculus.
For example \[
  \int_x^y \omega = F(y)-F(x)\;, \quad F = \int \omega\;.
\] 

Both Coleman integration, in its Tannakian interpretation~\cite{Bes99}, and
Vologodsky integration~\cite{Vol01}, have very similar constructions. Given
the space \( X \) one considers the category \( \Un(X) \) of unipotent
connections on \( X \). Each point \( x\in X(K) \) gives a fiber functor \(
\omega_x\) on \( \Un(X) \) associating with the connection the fiber of the
underlying vector bundle at \( x \). For each two points \( x,y\in X(K) \) one has a
Tannakian path space \( P_{x,y} \) of paths from \( x \) to \( y \), whose
\( A \)-valued points, for any \( K \)-algebra \( A \), are the
tensor isomorphisms between \( \omega_x \otimes A \) and \( \omega_y
\otimes A \). It turns out that this path space has a canonical (\( K
\)-linear) automorphism
\( \varphi \) (Frobenius) and, in the case of algebraic varieties over \( K
\), also a canonical ``monodromy'' \( N \), which is a derivation of the
associated algebra of regular functions. In both cases one proves
(\cite[Corollary~3.2]{Bes99} and~\cite[Theorem~B]{Vol01}) the existence of a canonical path
\( \gamma_{x,y} \in P_{x,y}(K) \). In both cases this path is Frobenius
invariant. In the overconvergent case this property completely
characterizes it. In the algebraic case this property is not sufficient and
one needs an additional property involving monodromy, which we do not
describe. The canonical paths are both invariant under composition, and
functorial in an obvious sense.

Following~\cite{Bes99} we can now define Coleman, respectively Vologodsky
functions. These are represented by fourtuples \( F=(M,\nabla ,v=v_x,s) \) where
\( (M,\nabla )\in \Un(X) \), \( v\in M_x \) and \( s: M\to \O_X \) is a
morphism from the underlying vector bundle. Here \( v_x \) is either chosen
at one fixed point \( x\in X(K) \) and transported via the canonical path
to all other points, or we can avoid a choice of a point by choosing at
each point in a way compatible with the \( \gamma \)'s. Obviously both ways
give the same thing. The function \( F \) can be evaluated at a point \( y
\), \[
  F(y) = s(v_y) = s(\gamma_{x,y}v_x)\;.
\] 
A morphism \( F\to F' \) is a morphism \( (M,\nabla )\to (M',\nabla') \)
compatible with \( s,v \) in the obvious way. If such a morphism exists,
the values of \( F \) and \( F' \) are identical at each point, and we
identify them. 

One key difference between Coleman and Vologodsky integration, is the
dependency of the latter on a branch of the \( p \)-adic logarithm. This is
manifested by having \( \gamma_{x,y} \) defined over the algebra \( K[\log
(p)] \) of polynomials in a formal variable \( \log (p) \) over \( K \). We
can fix the branch by evaluating at one particular value for \( \log (p)
\), but our interest is rather to study the dependency on \( \log (p) \).
This makes Vologodsky functions have values in \( K[\log (p)] \)
We will use the notation \( d / d \log (p) F \) to mean, more precisely,
the value of the derivative of the Vologodsky function \( F \) with respect
to \( \log (p) \) evaluated at \( \log (p)=0 \). Often, with \( \pi \) a
fixed uniformizer for \( K \), it will make more sense to consider the
derivative \( d / d \log (\pi) F\).

In the non-iterated case, Vologodsky integration can be described purely
using its functoriality properties. It was in fact defined
before~\cite{Zar96,Colm96}. The following is well known.
\begin{proposition}\label{abvar}
  If \( X \) is an abelian variety and let \( 0 \) be the identity element.
  Then, \( \int_{0}^x
  \omega\) is the logarithm of \( X \) to its Lie algebra,
  paired with the value of the invariant differential \( \omega \) at \( 0
  \). This suffices to determine \( \int \omega \).
  If \( X \) is a proper variety, and \( \alpha: X\to A=\alb(X) \) the map
  to its Albanese variety, then any \( \omega \in \Omega^1(X) \) is \(
  \alpha^\ast \omega' \) for some invariant differential \( \omega' \) on
  \( A \). and \( \int \omega = \alpha^\ast \int \omega' \).
\end{proposition}
This way of obtaining the integral is independent of \( \log (p) \).
Therefore, we obtain our second result on the derivative of Vologodsky
integrals with respect to \( \log (p) \).
\begin{corollary}[{\cite{BMS21}[Proposition~9.14]}]\label{hol-indep}
 Let \( \omega \) be a holomorphic form on a proper variety \( X \). Let \(
  x_0\in X(K)\). Then, \[
    \dlp \int_{x_0}^x \omega = 0\;.
  \]  
\end{corollary}
Both Zarhin and Colmez extend their (non-iterated) theory to varieties
which are non-proper. In this case, just as for the logarithm, there is
indeed a dependency on \( \log (p) \). These theories coincide with
Vologodsky's more general integration theory.

\section{Syntomic regulators}\label{sec:syntomic}

We approach syntomic regulators via \'etale regulators. This has the
advantage of being able to treat, for the type of situations we are
considering, varieties with arbitrary reduction, without the need for the
sophisticated machinery of~\cite{Nek-Niz14}. We stress that for proving
some important results about syntomic regulatos, as well as for
computations, one is forced to get a much closer understanding of their
theory.

We assume that \( X \) is a proper smooth variety over \( K \). For every
finite prime \( \ell \) we 
have \'etale regulator maps
\begin{equation*}
  r_\ell: \hm^i(X,\Q(j)) \to \het^i(X,\Q_\ell(j))\;,
\end{equation*}
where the cohomology on the right-hand side is continuous \'etale cohomology
as defined by Jannsen~\cite{Jan88}. By a theorem of
Jannsen~\cite[Theorem~1.1]{Ras95}
The Hochschild-Serre spectral sequence,
\begin{equation}  \label{eq:spectral}
  E_2^{r,s}= H^r(K,\het^s(X\otimes \Kbar,\Q_\ell(j))) \Rightarrow
  \het^{r+s}(X,\mathbb{Q}_\ell(j))\;,
\end{equation}
degenerates at \( E_2 \)  Let \( \hm^i(X,\Q(j))_0 \) be the kernel of the composition
\begin{equation} 
  \hm^i(X,\Q(j)) \to  \het^{i}(X,\mathbb{Q}_\ell(j)) \to 
   H^0(K,\het^i(X\otimes \Kbar,\Q_\ell(j)))\;.
\end{equation} 
We get an induced map
\begin{equation}\label{eq:etreg}
  \hm^i(X,\Q(j))_0 \to
   H^1(K,\het^{i-1} (X\otimes \Kbar,\Q_\ell(j)))\;.
\end{equation}
At this point we restrict attention to the case \( \ell=p \).
One important consequence of~\cite{Nek-Niz14} is that the
map~\eqref{eq:etreg} is actually
into the subgroup of so called semi-stable classes,
\begin{equation}\label{eq:syndef}
  \rsyn: \hm^i(X,\Q(j))_0 \to
  \hst^1(K,
   \het^{i-1} (X\otimes \Kbar,\Qp(j)))\;.
\end{equation} 
Here, $V:= \het^{i-1}(X\otimes \Kbar,\Qp(j))$ is regarded as a
potentially semi-stable representation of
$G=\Gal(\Kbar/K)$ and $\hst^1(K,V)$
is semi-stable cohomology which may be interpreted as the group of
extensions in the
category of potentially semi-stable representation and may be computed in
terms of the complex $\cst(V)$ of~\cite[1.19]{Nek93} which we recall
shortly. The map \( \rsyn \) in~\eqref{eq:syndef} is our syntomic regulator. Notice that so far
this map is independent of any choice of branch.
\begin{remark}\label{toric-remark}
  For use in Section~\ref{sec:toric} we point out that when \( \ell \ne p \) the
  \'etale regulator lands in the group \( H_g^1 \) of unramified classes.
  Furthermore, there is an integral version.
\[
 \reg_\ell: \hmot(X,i,j)_0 \to
   H_g^1(K,\het^{i-1} (X\otimes \Kbar,\Z_\ell(j)))\;.
\] 
\end{remark}

We now analyze more carefully the target of the syntomic regulator using
the
complex $\cst(V)$. At this stage, \( V \) can be any de Rham
representation. Recall the Fontaine functors
$\Dst$ and $\DR$. For the Galois representation \( V \), $\Dst(V)$ is a
$K_0$-vector space, equipped with a linear nilpotent operator $N$ (called
monodromy) and a semi-linear (with respect to the unique lift of Frobenius
on $K_0$ operator $\varphi$ (Frobenius), satisfying the relation
\begin{equation}\label{enfi}
  N\varphi= p \varphi N\;.
\end{equation}
\( \DR(V) \) is a \( K \) vector space with a descending filtration \(
F^{\bullet}  \), and these two functors are related by an isomorphism, \[
  I_\pi: \Dst(V)\otimes_{K_0} K \to \DR(V)\;,
\] 
which depends on a choice of a uniformizer \( \pi \) of \( K \). The
dependency of the uniformizer is such that if \( \pi' \) is another
uniformizer, then
\begin{equation}\label{chofuni}
  I_{\pi'}= I_\pi \circ \exp (\log (\pi' / \pi )N) 
\end{equation} 
The collection of data \( (D=\Dst(V),D_K = \DR(V), I_\pi) \) with the above
conditions form a category, called the category of \emph{filtered Frobenius
monodromy modules}, \( \ffm \). The subcategory of \( \ffm \) coming from
semi-stable Galois representations is the abelian category of
\emph{admissible}
filtered Frobenius monodromy modules. The category \( \ffm \) has a canonical object \(
\one \) corresponds to the trivial one dimensional representation and
objects \( D \in \ffm \) have Tate twists \( D(n) \) corresponding to
shifting the filtration by \( n \) and multiplying \( \varphi \) by \( p^{-n}  \).
The Fontaine functors give an obvious functor
\begin{equation}\label{Fonfun}
  \mathbb{D}: \operatorname{Rep}_{\mathbb{Q}_p}(\Gal(\overline{K} / K)) \to
  \ffm\;.
\end{equation} 

 For a de Rham representation $V$
The complex \( \cst(V) \)
of~\cite[1.19]{Nek93} computing $\hst^\bullet(K,V)$
is given as follows:
\begin{equation}\label{sscompt}
  D \xrightarrow{(\varphi-1,N,-I_\pi)}D\oplus D \oplus D_K/F^0
  \xrightarrow{N+1-p\varphi+0} D\;. 
\end{equation}
Similarly, we have a complex \( \csp(V) \) given by
\begin{equation}\label{sscompp}
  D \xrightarrow{(\varphi-1,N)}D\oplus D 
  \xrightarrow{N+1-p\varphi} D\;. 
\end{equation}
A short exact sequence in \( \ffm \) is a sequence of morphisms, which
gives an exact sequence on the \( D \)'s and a strict exact sequence on the
\( D_K \)'s. The collection of isomorphism classes of extensions of two
fixed such objects form a group under the Baer sum.
\begin{lemma}\label{Yoneda}
  The group \( H^1(\cst(V)) \) classifies Yoneda exts of \( \one \) by \(
  \mathbb{D}(V) \), where \( \mathbb{D} \) is the Fontaine functor
  from~\eqref{Fonfun}.
  Similarly, the group
   \( H^1(\csp(V)) \) classifies Yoneda exts of \( \one \) by \(
  \Dst(V) \).
\end{lemma}
\begin{proof}
  Let \begin{equation}\label{sample-ext}
  0\to \mathbb{D}(V) \to \mathbb{D}' \to \one \to 0
.\end{equation} 
  be an extension in \( \ffm \). We can find lifts, \( A\in D' \) of \( 1\in
  \one \) and \( B \in F^0 D_K^\prime \) of \( 1\in \one_K \),
  the latter since \( 1\in F^0 \one_K \) and the morphisms are
  strict. We associate with the pair \( (A,B) \) the triple
  \[
  (x,y,z)=((\varphi-1)A,NA,B-I_\pi(A))\in D\oplus D\oplus D_K /F^0\;.
  \]
  We quotient
  by \( F^0 \) to account for the fact that \( B \) is only determined up
  to an element of \( F^0 D_K \). This element is clearly
  in the kernel of \(  
  N+1-p\varphi 
  \).  When we add to  \( A \) an element \( w\in D \) we add to \( (x,y,z)
  \) the element  \( 
  (\varphi-1,N,-I_\pi)(w)
  \). Thus, an extension gives a well-defined element of \( H^1(\cst(V)) \).
  Conversely, it is easy to construct an extension from such an element.
  The second statement follows similarly (and more easily).
\end{proof}
\begin{proposition}
  Suppose \( D^{\varphi =1, N=0}=0  \). Then
 we have a short exact sequence,
  \begin{equation}\label{eq:shorty}
    0 \to D_K / F^0 \xrightarrow{\myi} H^1(\cst(V)) \xrightarrow{\myp}
    H^1(\csp(V)) \to 0\;,
 \end{equation}  
  given by \( \myi(z)=(0,0,z) \) and \( \myp (x,y,z)= (x,y) \). 
\end{proposition}
\begin{proof}
This is more or less clear. The only point to note is that since 
  \( D^{\varphi =1, N=0}  \) is the kernel of the first two coordinates of
  the first differential, the
  assumptions imply that an element \( (0,0,z) \) can not be in the image
  of the first differential.
\end{proof}
\begin{remark}
  It is easy to interpret the short exact sequence~\eqref{eq:shorty} as follows:
  The projection forgets the \( D_K \) data leaving an extension of
  Frobenius monodromy modules. The injection gives all possible filtration
  on the trivial extension of Frobenius monodromy modules.
\end{remark}
In order to split the short exact sequence~\eqref{eq:shorty} we now restrict
attention to $V:= \het^{i-1}(X\otimes \Kbar,\Qp(j))$ for \( X \) with
semi-stable reducton. 
By the semi-stable conjecture of Fontaine, proved by Tsuji~\cite{Tsu99}
the associated \( D=\Dst(V) \) is isomorphic to Hyodo-Kato cohomology of
the special fiber (viewed as a log-scheme) while \( D_K = \DR(V) \) is isomorphic to the de Rham
cohomology \( \hdr^{i-1}(X /K)  \), both twisted by \( j \). Consequently, 
by~\cite{Mok93}
there is a weight decomposition, \[
    D= \oplus_k D^k\;,
\]
where the linear
Frobenius $\phi=\varphi^f$ operates on $D^i$  with eigenvalues which are
Weil numbers of weight $i$. The relation~\eqref{enfi} immediately implies
that \( N D^i \subset D^{i-2} \).
\begin{proposition}\label{twocansplit}
  Suppose either \( D^0=0 \) or that \( N:D^0 \to D^{-2}  \) is an
  isomorphism. Then, the sequence~\eqref{eq:shorty} canonically splits and
  we have \( H^1(\csp(V)) \isom D^{-2}  \) if \( D^0=0 \) or \( 0 \) if \( N \) is an
  isomorphism.
\end{proposition}
\begin{proof}
Consider a triple \( (x,y,z) \) representing a class in \( H^1(\cst(V))
  \). It satisfies the relation \( Nx = (p\varphi -1) y \). Decompose \[
    x = \sum x_k, \quad x_k \in D^k\;.
  \] 
It is easy to see, by weight considerations, that \( \varphi-1 \) is
  invertible on \( D^k \) when \( k\ne 0 \). Let then,
for \( k\ne 0 \), \( w_k = (\varphi-1)^{-1}x_k  \)
  Suppose first that \( D^0=0 \). Then, \( x_{0}=0 \) and we see that by subtracting the
  differential of \( w=\sum w_k \), we can make \( x=0 \). This makes \(
  H^1(\cst(V) \) isomorphic to \( H^1 \) of the complex \[
  D^{\varphi=1}  \xrightarrow{(N,-I_\pi)} D \oplus D_K/F^0
  \xrightarrow{1-p\varphi+0} D
  \] 
  and, since \( D^{\varphi=1}\subset D^0=0  \), to just \( D^{p\varphi = 1}
  \oplus D_K / F^0  \) and the map \( \myi \) is just the embedding of the first
  factor, and is clearly canonically split. Suppose now that \( N: D^0 \to
  D^{-2} \) is an isomorphism. Since we are not assuming that \( D^0=0 \)
  we can only assume that \( x\in D^0 \), hence \( y\in D^{-2}  \). By the
  assumption on \( N \) we can now modify by the differential of \( w' \)
  to make \( y=0 \), but then \( N x =0 \) so that \( x=0 \), again by
  assumption. This proves the result in the second case.
\end{proof}
\begin{remark}\label{forproof}
  In terms of of the proof of Lemma~\ref{Yoneda} the splitting in the case
  \( D^0=0 \) is given as follows: Given an extension \( \varepsilon \) as
  in~\eqref{sample-ext}
    Lift \( 1\in \one
  \) to the unique \( A\in D' \) which is Frobenius invariant (this is
  equivalent to having \( x=0 \) in the proof of Proposition~\ref{twocansplit}). Then send the
  extension to \( B- I_\pi(A) \) as before. Note that since \( A \) is
  Frobenius invariant we have \( \myp(\varepsilon)= N A \)
\end{remark}
\begin{definition}\label{tworegdef}
  Suppose \( D^0=0 \) and let \( \mysplit_\pi :\hst^1(G,V)\to D_K / F^0 \) be the canonical
  splitting of~\eqref{eq:shorty}. For a
  variety \( X \), with
$V= \het^{i-1}(X\otimes \Kbar,\Qp(j))$,
let
\begin{align*}
  \rsyn^{c,\pi} &:
   \hm^i(X,\Q(j))_0 \to D_K / F^0 \\
  \rsyn^{d} &:
   \hm^i(X,\Q(j))_0 \to D^{-2} \;,
\end{align*}
called the continuous and discrete
  component of the syntomic regulator respectively, be the compositions of the syntomic
  regulator \( \rsyn \) from~\eqref{eq:syndef} followed by the splitting \(
  \mysplit_\pi \) and the projection \( \myp \) respectively.
\end{definition}
\begin{theorem}\label{synderi}
  For any extension \( \varepsilon \) we have \[
    \frac{d}{d \log (\pi)} \mysplit_\pi (\varepsilon) = - I_\pi\circ \myp
    (\varepsilon)\;.
  \] 
Consequently, we have
  $ \frac{d}{d \log (\pi)} \rsyn^{c,\pi} =
  - \rsyn^{d}$ 
\end{theorem}
\begin{proof}
  Following Remark~\ref{forproof} we see that, in the notation there, \[
    \frac{d}{d \log (\pi)} \mysplit_\pi (\varepsilon) = 
    \frac{d}{d \log (\pi)} B- I_\pi (A)
    =- \frac{d}{d \log (\pi)}  I_\pi (A)\;,
  \] 
  and~\eqref{chofuni} immediately shows that \[
    \frac{d}{d \log (\pi)}  I_\pi (A) = I_\pi N A \;.
  \] 
  Finally, by Remark~\ref{forproof} again, we have \( N A =
  \myp(\varepsilon) \).
\end{proof}

There are two cases where it is known that syntomic regulators in the bad
reduction case is given by Vologodsky integration. The first
is that of zero-cycles on a proper variety \( X \) of
dimension \( n \). In this case we have \( j=n, i=2n \) and the regulator
takes the form of a map \[
  CH^n(X)_0 \to 
  \hst^1(K,
   \het^{2n-1} (X\otimes \Kbar,\Qp(n)))\;.
\] 
Note that in this case \( D^0 \) and \( D^{-2}  \) for \( V \) correspond to
weights \( 2n \) and \( 2n-2 \) of \( \Dst(\het^{2n-1}(X\otimes
\overline{K},\mathbb{Q}_p) ) \). Thus, if the monodromy weight
conjecture holds for \( X \), then \( N: D^0 \to D^{-2}  \) is an
isomorphism. Therefore, by Proposition~\ref{twocansplit}, the discrete component of
the regulator is \( 0 \). The syntomic regulator is therefore equal to its
continuous component \[
  \rsyn: 
  CH^n(X)_0 \to \hdr^{2n-1}(X /K) 
  / F^n\;.
\]
The target of the regulator is Poincare dual to \( F^1 \hdr^1(X / K) \),
the space of holomorphic one forms on \( X \).
The following is well known. 
\begin{proposition}
  The vector space \( CH^n(X)_0 \) consists of zero cycles of degree \( 0
  \) on \( X \), and for such a cycle \( z \)  and a holomorphic form \(
  \omega \) one has \[
    \rsyn(z)(\omega) = \int_z \omega\;.
  \] 
\end{proposition}
As we saw in Corollary~\ref{hol-indep}, the Vologodsky
integral of a holomorphic form \( \omega \) on a proper \( X \) is
independent of the choice of the branch of the logarithm. This is
consistent with the syntomic regulator having no discrete part in this
case.

The second case
is the
regulator for \( K_2 \) of curves. In~\cite{Bes18} we proved the following
result.
\begin{theorem}[{\cite[Corollary~1.3]{Bes18}}]\label{k2sstwo}
 Let \( X \) be a proper curve with arbitrary reduction over \( K \). Then,
  with respect to a fixed uniformizer \( \pi \), 
  the result of Theorem~\ref{ktwogood} continues to hold provided we
  replace the syntomic regulator with its continuous part with this choice of uniformizer, assume that the
  form \( \omega \) is in the kernel of the monodromy operator
  and replace Coleman
  integration with Vologodsky integration with respect to the branch of
  logarithm determined by \( \pi \).
\end{theorem}
This is a special case of a more general result where one does not assume
that the form \( \omega \) is holomorphic, but only that it is in the
kernel of monodromy. We will say more about this generalization in
Section~\ref{sec:toric}.
The monodromy operator \( N \) is transported from \( \Dst \) to \( \DR \)
via the comparison isomorphism. We will
  look at this operator for the special case of curves later in the paper.

\section{The derivative of a Vologodsky function with respect to \( \log (p) \)}
\label{sec:deriv}

In this section we study the universal log variant of Vologodsky
integration, as introduced in Section~\ref{sec:colvol}, and its derivative with
respect to \( \log (p) \).

The simplest result on this derivative was already mentioned in the
introduction. The universal logarithm \( \log  \) on \( \gm \). We proved in the introduction that \[
  \frac{d}{d \log (p)} \log (z) = v_p(z)\;.
\] 

Before continuing it is worth pointing out the obvious functoriality
property,
\begin{equation}\label{eq:fun}
 \dlp (\alpha^\ast F) = \alpha^\ast \left(\dlp F \right) \;,
\end{equation} 
valid for any Vologodsky function \( F \) on \( X \) and any morphism \(
\alpha: Y \to X \)~\cite[Lemma~9.13]{BMS21}. In particular, for any rational function on \( X \) we
get \[
  \left(\dlp \log (f)\right)(x) = v_p(f(x)) \;.
\] 

There is one particular version of Coleman integration that has some
dependency on the branch of the logarithm. Let \( X \) be the generic fiber
of a smooth proper \( \O_K \)-curve \( \mathcal{X} \) and let \( T=
\mathcal{X}_\kappa \). There is a reduction map \( \red: X\to T \), where
we abuse notation to write \( X \) for its rigid analytification. the
preimages of \( \kappa \)-rational points in \( T \) are rigid analytic
open discs in \( X \) called \emph{residue discs}. Let \( S\subset
T(\kappa) \) be a finite subset and let us, for each \( y\in S \), pick a
strictly smaller
rigid analytic disc \( D_y\subset \red^{-1}(y)  \). The space \(
U=X-\cup_{y\in S} D_y \) is a wide open space, and the underlying affinoid
is \( A = X- \red ^{-1}(S)  \). Their difference is a disjoint union of annuli,
\( e_y = \red ^{-1}(y)-D_y  \). if \( \omega \) is a rigid one form on \( U
\), its Coleman integral, which by standard theory is only defined on \(
A \), extends to the annuli \( e_y \). On these annuli \(
\omega \) is given by a convergent Laurent series and integration is done
term by term, with the residual term integrating to \( (\res_{e_y}
\omega)\cdot
\log(z) \), with \( z \) some local parameter. Note that multiplying the
parameter by a constant will change the constant of integration, which is
anyhow a degree of freedom. To fix the constants of integration, the
Tannakian interpretation of Coleman integration is extended to fiber
functors associated with the \( y \)'s~\cite[Section~5]{Bes99}. The theory
easily extend to meromorphic forms on \( U \) by successively removing
smaller and smaller discs around the singular points, and we have a similar
local expansion: If \( x \) is a singular point for \( \omega \) and
\( \res_x
\omega \ne 0 \), then the integral has a logarithmic term \( (\res_x
\omega)
\log (z_x)\), for a local parameter \( z_x \) at \( x \), and this will
clearly depend on \( \log (p) \). By the above, we see that if \( \res_x
\omega = 1 \) and \( \pi \) is a uniformizer for \( K \),
\begin{equation}\label{inters}
  \left(\frac{d}{d \log (\pi)} \int \omega\right) (x') = v_\pi(z_x(x'))\;.
\end{equation}  
This almost proves the following.
\begin{lemma}\label{lastlemma}
  Suppose \( D =\sum n_i P_i \) is a divisor on \( X \) with \( P_i \in A(K)
  \), with \( A \) the
  underlying affinoid of a wide open space \( U\subset X \). Let \( x_0\in A(K) \) be a point whose
  reduction is disjoint from the reduction of the points \( P_i \) and
  let \( \omega \) be a meromorphic form on \( U \), regular on the residue
  disc of \( x_0 \), whose residue divisor
  is \( D \) and choose \( \int  \omega \) to be a Coleman integral
  vanishing on \( x_0 \). Let \( P\in A(K) \). Then \[
    \frac{d}{d \log (\pi)} \int \omega (P) = P \cdot D\;,
  \] 
  the intersection multiplicity of \( P \) with \( D \) on \( \mathcal{X}
  \), where \( P \) and \( D \) are considered divisors on \( \mathcal{X}
  \) by extending them to section and considering the associated sections.
\end{lemma}
\begin{proof}
  By linearity we may assume that \( D=P_1 \) is a single point. Since the
  integral of a holomorphic form does not depend on the branch, it does not
  matter which form \( \omega \) we pick. By removing additional residue
  discs we may assume that \( \omega = d \log (f) \) for a rational
  function \( f \) which is regular on \( U-P_1 \), with \( f(x_0)=1 \).
  This implies that \( T \) is not in the divisor of \( f \) and clearly \(
  P_1 \cdot P = v_{\pi} (f(P))\). The result 
  follows from~\eqref{inters}.
\end{proof}

Results for more general integrals are known only for curves with
semi-stable reduction, as this is the only case where the relation between
Vologodsky integration and Coleman integration is known, first, in the
non-iterated case by~\cite{Bes-Zer13}, and then, in general, by~\cite{Kat-Lit21}.

We recall the setup of~\cite[Section~4]{Bes17}.
We assume that the curve $X$
is the generic fiber of a proper regular flat $\O_K$ scheme $\XX$ of relative dimension
$1$ with
semi-stable reduction $T$ which decomposes into smooth irreducible components
\begin{equation}
  \label{eq:Yi}
  T = \cup_i T_i\;.
\end{equation}
For simplicity we will assume that components $T_i$ and $T_j$ intersect at
at most one point.

Let $\Xgrph$ be the dual graph of $T$ with vertices $V$ and edges $E$
(this is of course an abuse of notation as it really depends on the
particular model). The vertices correspond to the components $T_v$
while the edges are ordered pairs of intersecting components
$(T_v,T_w)$ oriented from $v$ to $w$, so that an edge $e$ has tail
$e^+=v$ and head $e^-=w$.

The reduction map $\red: X\to T$ allows us to cover $X$ with rigid analytic
domains $U_v = \red^{-1} T_v$ which are basic wide open spaces. These then intersect along annuli corresponding
bijectively to the unoriented edges of $\Xgrph$ (see Figure~\ref{figureone}) .
An orientation of an annulus fixes a sign for the residue along this
annulus, and we match oriented edges with oriented annuli as
in~\cite[Definition~4.7]{Bes17}. We use the same notation for
the edge and for the associated oriented annulus.

\begin{figure}
\includegraphics[width=\textwidth,height=1.875in]{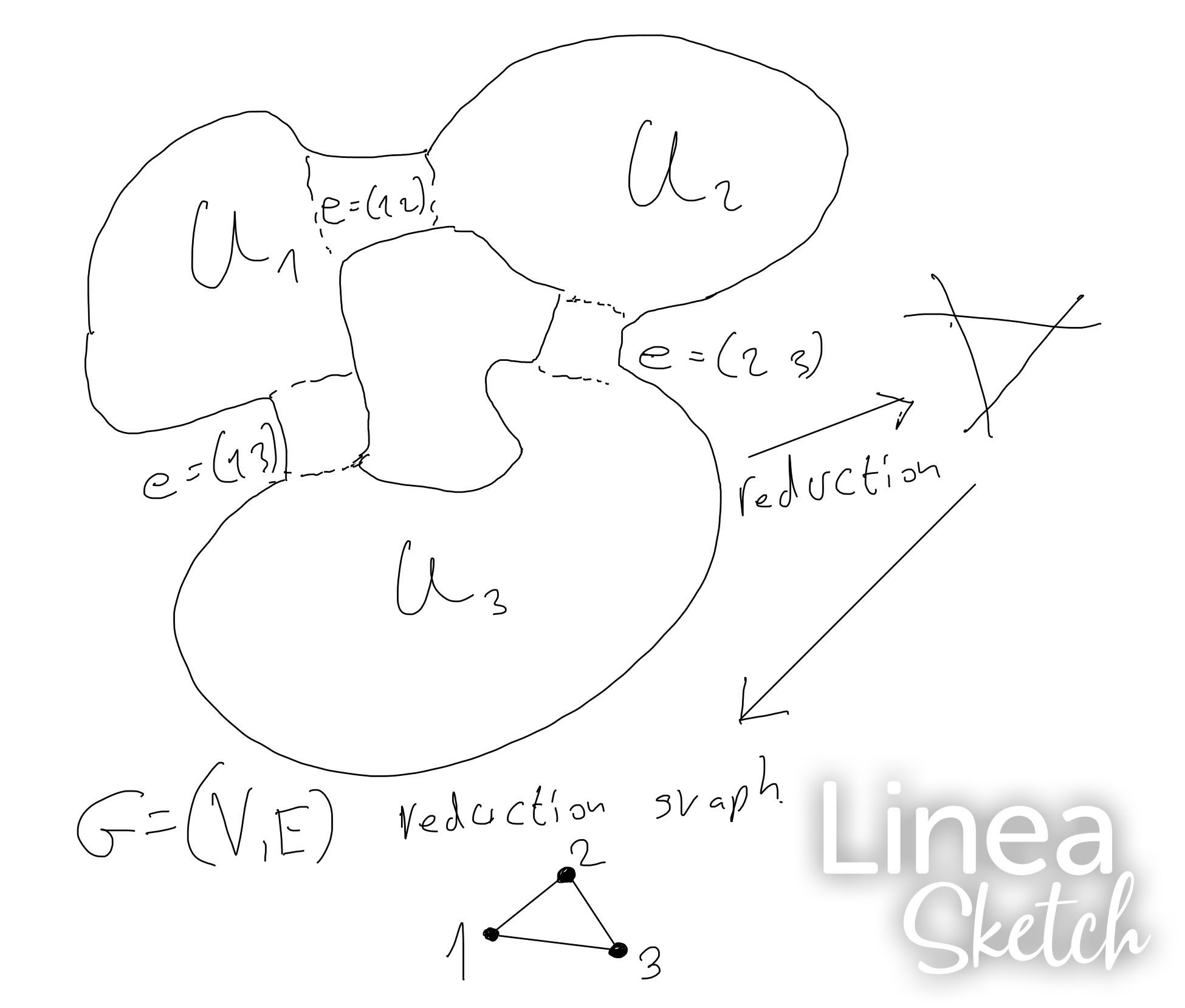}\\
\caption{semi-stable curve and its reduction graph.}
\label{figureone}
\end{figure}

We give a 
Crash course on graph cohomology. All of this material is well known.
For us, a graph \( \Gamma \) has a set \( V \) of vertices and \( E \) of
oriented edges.
Each edge
$e\in E(\Gamma)$ has a tail and head, $e^+,e^- \in V(\Gamma)$
respectively, and for each edge we have the edge with reverse
orientation $-e$ such that $(-e)^+ = e^-$ and $(-e)^- = e^+$
and \( K \) any field, let \( K^V \) denote as usual the
vector space of functions from \( V \) to \( K \), and \( K_-^E \) the set
of functions \( E\to K \) which are antisymmetric with respect to
orientation-reversing: \( f(-e)=-f(e) \). There is a differential
\(d: K^V\to K_-^E\), \(df(e)=f(e^+)-f(e^-)\) and we define graph cohomology
by
\begin{equation*}
H^1(\Gamma,K)= K_-^E / d K^V\;.
\end{equation*}
Both \( K^V \) and \( k_-^E \) have obvious scalar products and 
the dual operator to \( d \) with respect to these, \(d^\ast: K_-^E \to
K^V\), is given by
\(d^\ast g (v)= \sum_{e^+ = v} g(e)\).
The space of 
\emph{harmonic cochains} is defined as \(\mathcal{H}(\Gamma,K)= \Ker d^\ast\).

For a connected \(\Gamma\) the harmonic decomposition
\begin{equation}\label{harm-dec}
K_-^E = d K^V \oplus \mathcal{H}(\Gamma,K)
\end{equation}
follows by observing that
the kernel of the \emph{Laplacian}
\(\Delta=d^\ast \circ d: K^V \to K^V\) is the space of constant
functions and its image is its orthogonal complement. We note the following
easy observation.
\begin{lemma}\label{easyobs}
The projection on \(
d K^V\) coming from the harmonic decomposition~\eqref{harm-dec} is simply
\[
d \circ \Delta^{-1}\circ d^\ast\;.
\]
\end{lemma}

One source of harmonic cochains is the monodromy operator \( N \). For a
meromorphic form on \( X \) with no singularities on the annuli it is
defined as \[
  N\omega (e) := \res_e(\omega)\;,
\] 
where the residue of an analytic form on an oriented annulus is defined
in~\cite[Lemma~2.1]{Col89}. 
The residue theorem for wide open spaces~\cite[Proposition~4.3]{Col89} immediately shows
  that
  \begin{equation}\label{nharmonic}
  d^\ast N \omega (v)= \res_{U_v} \omega\;,
  \end{equation}
 where the right-hand side means the sum of the residues of \( \omega \) on
  \( U_v \). Recall that a meromorphic form on a curve is said to be of the
  \emph{second kind} if all its residues are \( 0 \). For such a form \(
  \omega \) we immediately get from~\eqref{nharmonic} that
\( N \omega \)
is harmonic. The monodromy operator as defined here is essentially the same
as the one coming from Fontaine's theory. Indeed, by~\cite{ColIov99} the
space \(  \mathcal{H}(\Gamma,K)
 \) embeds into \( \hdr^1(X /K) \) and the Fontaine monodromy
operator is just the composition of \( N \) as defined here with this
embedding. 

The results of~\cite{Bes-Zer13} are for a Vologodsky integral with respect
to a fixed branch, so at this point we fix such a branch. To state the main result
of~\cite{Bes-Zer13} we point out that as the domains \( U_v \) are basic
wide open spaces with good reduction, forms on them can be Coleman
integrated. These integrals extend, once a branch of the logarithm is
fixed, to the annuli \( e \). This is not part of the general theory and
needs to be argued separately (see~\cite[Section~5]{Bes99}).
\begin{theorem}[\cite{Bes-Zer13}]\label{beszer}
  Let $\omega$ be a meromorphic form on $X$ and let $F_\omega$ be a
  Vologodsky integral of $\omega$. Then, for each vertex $v$ of
  $\Xgrph$ there exist Coleman integrals $F_\omega^v$ of
  $\omega|_{U_v}$ such that $F_\omega $ equals $F_\omega^v$ on
  $U_v(K)$. For an oriented edge $e$ let
  $c_\omega(e)=F_\omega^{e^-}|_{e} - F_\omega^{e^+}|_{e} \in K$. Then $c_\omega$ is
  a harmonic cochain on $\Xgrph$ with values in $K$,
  $c_\omega\in \hh(\Xgrph,K)$.
\end{theorem}
It is important to observe that \( F_\omega \) is \emph{not} a function on
the annuli. It is only a function on \( K \)-rational points and the
annuli do not have any such points. This is why the single valued function
\( F_\omega \) can be defined on the \( U_v \) by Coleman functions that do
not match on the intersections. Obviously, one can extend \( F_\omega \) to
the annuli by extending the base field, but this will force a change in the
semi-stable model and no contradiction arises (see~\cite{Bes-Zer13}).

Theorem~\ref{beszer} gives a way of recovering Vologodsky integrals on
curves, provided one knows how to compute Coleman integrals. Namely, to
compute \( F_\omega \) first choose Coleman integrals \( G^v \) for \(
\omega \) on the corresponding \( U_v \). There is a choice of a constant
of integration for each \( v \) and one makes such a choice arbitrarily.
Then one computes the cochain \( c(e)= G^{e^-}-G^{e^+}   \). This cochain
  has a unique decomposition, by~\eqref{harm-dec}, \[
  c = c_\omega + d \gamma \;,\quad c_\omega \text{ harmonic, }\gamma\in K^V\;,
\] 
where \( \gamma \) is uniquely determined by \( d\gamma \) up to an
additive constant. The required \( F_\omega^v \) computing a Vologodsky
integral of \( \omega \) are then clearly given by \[
  F_\omega^v = G^v - \gamma(v)\;.
\] 
Note that by the discussion above we have
\[
  \gamma = \Delta^{-1} d^\ast c\;.
\]
To give a flavor of the type of results one might prove using this type of
analysis, let us begin by reproving Corollary~\ref{hol-indep} (The reader might find it
useful to look first at the very simple case of a Tate elliptic curves
considered in Section~3 of~\cite{Bes-Zer13}).
The key to understanding how the description of Vologoksky integration
affects the dependency on \( \log (p) \) is to analyze what happens at an
annulus. Suppose an annulus \( e \) is mapped, via a coordinate \( z \), to
the annulus \( |\pi|< |z| < 1 \) and a form \( \omega \) is given, whose
local expansion on \( e \) is \[
  \omega = \sum_k a_k z^k \frac{dz}{z}
\] 
so that \( \res_e \omega = a_0 \)
The integral of \( \omega \) on the domain which connects to the inner side
of \( e \) is given by integrating term by term \[
  F_1 = \sum_{k\ne 0} \frac{a_k}{k} z^k + a_0 \log (z)+C_1
\] 
for some constant \( C_1 \). To get the integral on the domain connecting
to the outer side we first make the change of variables \( z= \pi / w \),
\[
  \omega = -\sum_k a_k (\frac{\pi}{w})^k \frac{dw}{w}\;,
\] 
giving an integral \[
  F_2 = \sum_{k\ne 0} \frac{a_k \pi^k}{k} w^{-k}  - a_0 \log (w)+C_2\;.
\] 
Substituting back \( z = \pi / w \) we get \[
  F_1-F_2 = C_1-C_2 + a_{0} \log (\pi)
  = C_1-C_2 + \res_e \omega \log (\pi)\;.
\] 
Suppose that \( F \) is a Vologoksky integral of \( \omega \) with respect
to a certain branch \( \log (\pi) \). On each \( U_v \) it is represented
by a Coleman
integral with a choice of constant of integration, and the chain of
differences \( c_\omega(e) \) is harmonic. It is better to write here \( c_{\omega,\log (\pi)} \) to stress the dependency on \( \log (\pi) \)  According to the above computation,
for the same choices of constants of integration, with respect to a
different branch of the logarithm \( \log' (\pi) \), the chain of
differences is
\begin{equation}\label{cdeplog}
  c_{\omega,\log'(\pi)} =
  c_{\omega,\log(\pi)} + (\log' (\pi)-\log (\pi))N \omega\;.
\end{equation}
But since
\( N\omega \) is harmonic, this chain is harmonic as well, hence the same
constants of integration give again the Vologoksky integral. Therefore, the
integral is independent of the choice of the branch.

The advantage of this proof is that it applies verbatim for integrals of forms of the
second kind, which are therefore also independent of the branch of the
logarithm. For general meromorphic forms we have the following easy
generalization.
\begin{proposition}\label{vol-mer}
  Let \( \omega \) be a meromorphic form on \( X \) with no singularites on
  the annuli. Let \( F_\omega \) be
  a Vologodsky integral of \( \omega \) with respect to the universal
  branch. Then, the function \( d / d \log (\pi) F_\omega \) is constant on
  each domain \( U_v \)
  except near the logarithmic singularities of \( \omega \), and
  associating this constant with the vertex \( v \) one gets a function on
  the dual graph which satisfies \[
   \Delta \frac{d}{d \log (\pi)} F_\omega (v) = \res_{U_v} \omega\;.
  \] 
\end{proposition}
\begin{proof}
  The same analysis as before together with Lemma~\ref{easyobs} gives that the derivative is
  \[
    - \Delta^{-1}  d^\ast N \omega 
  \] 
  and we are done by~\eqref{nharmonic}.
\end{proof}

Note that the above result determines the derivative only up to an additive
constant. This is unavoidable since one could in theory have a different
constant of integration associated with \( F_\omega \) for each branch of
the logarithm. To avoid this make the integral independent of \( \log (\pi)
\) at some point, e.g., by setting it to \( 0 \) at that point. However,
there are cases where we have the constant built in. For example, if \( f \)
is a rational function on \( X \) and \( \omega = d \log (f) \), then it is
natural to pick \( F_\omega = \log (f) \) and this might have a constant of
integration which is a multiple of \( \log (\pi) \). In fact, we clearly
have \[
  \frac{d}{d \log (\pi)} \log (f) (v) = \ord_{T_v} f\;.
\] 
Note that this is consistent with~\cite[Lemma~2.1]{Bes-Zer13}.

As a particular case, and as a prelude to the discussion in
Section~\ref{ss-loch}, let us use Proposition~\ref{vol-mer} to recover local height
pairings on curves.
In~\cite{Bes17} we extended the Coleman-Gross \( p \)-adic height pairing on
curves~\cite{Col-Gro89} using Vologodsky integration instead of Coleman
integration at primes above \( p \),  and proved its compatibility with the
Nekov\'a\v{r} \( p \)-adic height pairing. When \( X \) is a smooth complete
curve over a number field \( F \), and  \( D,E \) are two divisors of degree \( 0 \)
on \( X \), their \( p \)-adic height pairing \( h(D,E) \) is a sum of local terms
corresponding to the different finite completions of \( F \). if \( K \) is
such a completion, with respect to a prime above \( p \), the local term
for \( K \) is the trace down to \( \mathbb{Q}_p \) (we do not discuss this
trace here, see Section~\ref{ss-loch}), of the expression \[
  (D,E) = \int_E \omega_D\;.
\] 
Here, the integral is a Vologodsky integral evaluated at the divisor \( E
\) while the form \( \omega_D \) is a specially chosen meromorphic form of
the third kind
whose residue divisor \( \sum_x \res_x(\omega)\cdot x \) is \( D \) (recall that a meromorphic one-form on a
curve is of the \emph{third kind} if it only has simple poles and all of
its residues are integers). The determination of a canonical such \(
\omega_D\) is one of the key elements in the construction, but fortunately
for us, we do not need to discuss it here, as we will only be interested in
\[
   \frac{d}{d \log (\pi)} (D,E)=
 \frac{d}{d \log (\pi)} \int_E \omega_D\;,
\] 
and, since \( \omega_D \) is well-defined up to the addition of a
holomorphic form, this quantity is independent of the precise choice by
Corollary~\ref{hol-indep}.
\begin{proposition}\label{ascolmez}
  The expression 
\[
   \frac{d}{d \log (\pi)} (D,E)
\] 
  equals, up to a fixed multiple, the local height pairing as defined
  in the proof of ~\cite[Proposition~1.2]{Col-Gro89}. 
\end{proposition}
\begin{proof}
  For simplicity, we assume that \( D \) and \( E \) are formal sum of \( K
  \)-rational points and skip the formal argument to show that this is
  sufficient.
  The relevant, ``rational part'', of the local height pairing is defined
  in~\cite{Col-Gro89} as the intersection product on a regular model \(
  \mathcal{X} \) of \(
  X\), of divisors extending the divisors \( D \) and \( E \), to divisors
  which have a zero intersection with each component of the special fiber.
  In other words, we extend the points in \( D \) and \( E \) to sections
  and then add rational multiples of the components of the special fiber to
  satisfy the requirement. Clearly, it suffices to do this for only one of
  the divisors. So let \( D \) and \( E \) denote also the divisor on
  $\mathcal{X}$
  obtained by extending the points in \( D \) and \( E \) to sections over
  \( \O_K \) and let \( \tilde{D} = D+ \sum_v a_v T_v\) be the required
  extension, so that we need to compute
  \begin{equation}\label{locheight}
    \tilde{D}\cdot E = D \cdot E + \sum_v  a_v D\cdot T_v\;.
  \end{equation}
  Using Lemma~\ref{lastlemma} it is easy to see that we only need to account for the second term on the
  right-hand
  side of~\eqref{locheight}
  Clearly, \( \res_{U_v}\omega_D = \deg_{U_v}D \), where the right-hand
  side indicates the degree of the part of \( D \) inside \( U_v \).
    If \( f=\Delta^{-1}  (v\mapsto \deg_{U_v}D)\) then we easily see that
    the ``global'' contribution to the derivative is
  \[
     \sum_{v\in V} \deg_{U_v}E f(v)= E \cdot \sum f(v)T_v
  \] 
  and we are left with showing that \( f(v)=a_v+ \textup{const} \). The condition on \(
  \tilde{D} \) is that for every \( w\in V \) we have
  \[
    0= (D+\sum_v a_v T_v)\cdot T_w = \deg_{U_w} D + \sum_v a_v T_v \cdot T_w\;.
  \]  
  The sum on the right is clearly only over \( v \) for which there exist
  an edge \( (v,w) \), and over \( v=w \). Our assumptions about the
  reduction imply that for \( v\ne w \) in the sum we have \( T_v \cdot T_w
  = 1 \). Furthermore, since the special fiber is just \( \sum_v T_v \) and its
  intersection with each \( T_w \) is \( 0 \), we see that
  \[
  T_w \cdot T_w = - \sum_{(v,w)\in E} 1
  \]
  so that we have
  \[
    0=  \deg_{U_w} D + \sum_{(v,w)\in E} a_v - a_w\;.
  \]  
  This gives on \( v\to a_v \) the same condition as the condition on \( f
  \), making them equal up to a constant as required.
\end{proof}
\begin{remark}\label{colmremark}
  As mentioned in the introduction, a similar observation has been made by
  Colmez~\cite{Colm96}, in the proof of
Th\'eor\`eme~II.2.18.
  The method is quite different.
  Colmez simply notes that the derivative satisfies the basic list of
  properties which uniquely characterize the height pairing. Here on the
  other hand, we directly show that the result is the same as the
  construction of the local height pairing in terms of intersection theory.
\end{remark}

In order to treat iterated integrals we need the work of Katz and
Litt~\cite{Kat-Lit21}. We will present it in an equivalent way (the
equivalence with the original formulation was proved by Katz in a private
communication). The new formulation is a direct generalization
of~\cite{Bes-Zer13}.

To describe the result, choose for simplicity a \( K \) rational point \(
x_v \) in
each domain \( U_v \). Let \( (M,\nabla) \) be a unipotent connection on \(
X\). We have ``monodromy matrices'' corresponding to the following ``loops''
associated with each edge \( e=(v,w) \): Start with \( x_v \), use
Vologodsky parallel transport to go to \( x_w \), use Coleman parallel
transport to go to the side of the annulus \( e \) on the \( U_w \) side,
cross to the \( U_v \) side of \( e \) and finally Coleman parallel
transport back to \( x_v \) (see Figure~\ref{figuretwo}).

\begin{figure}
\includegraphics[width=\textwidth,height=1.77083in]{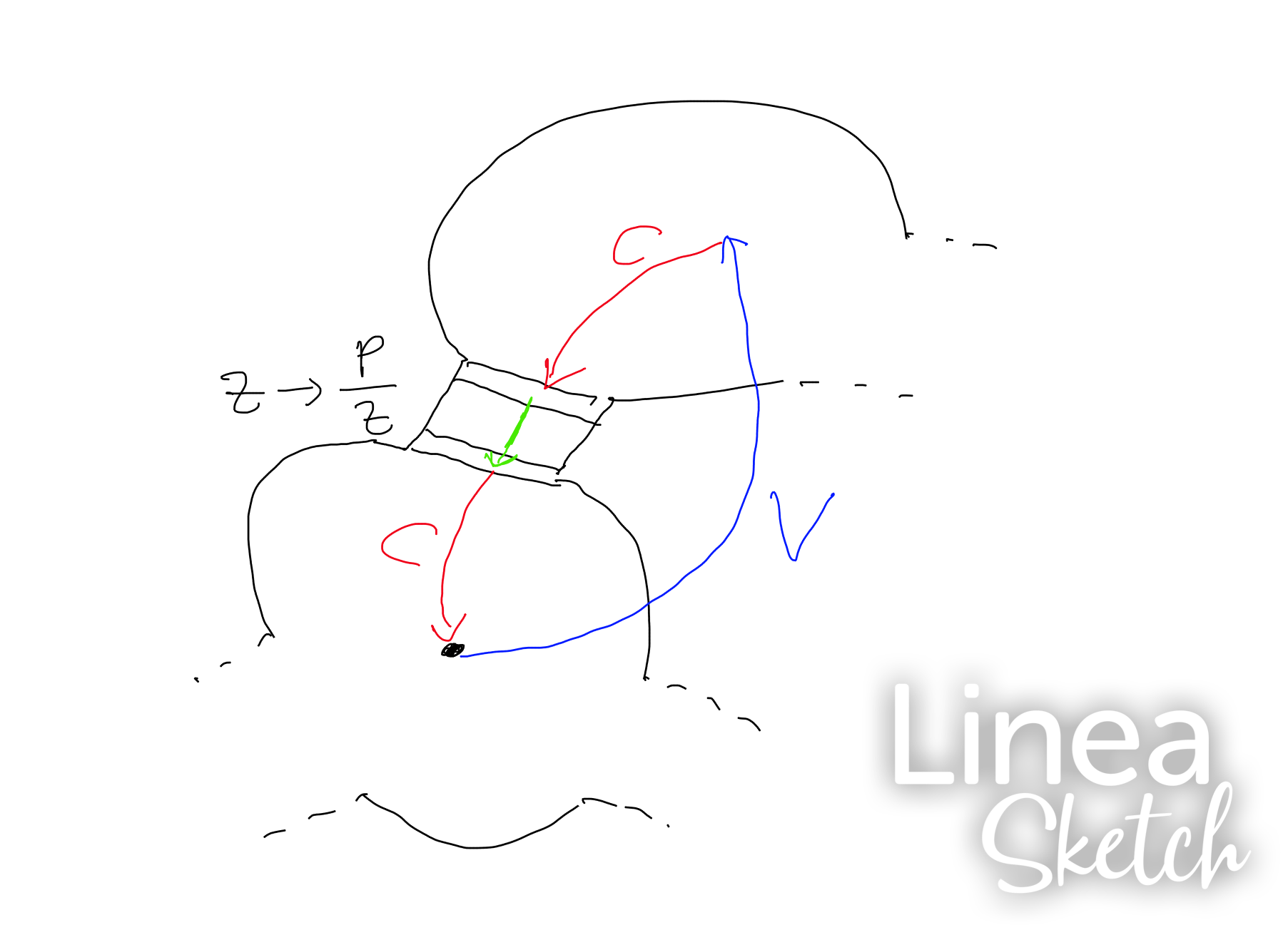}\\
  \caption{The loop determining \( A(e) \).}
\label{figuretwo}
\end{figure}

This gives an operator on the fiber at \( x_v \) of \( M \). Up to
conjugacy it is independent of the choice of \( x_v \). Denote it by \(
A(e) \).

\begin{theorem}[Reinterpreted Katz-Litt]
The association \(e\to A(e)\) is harmonic in the sense that for each
\(v\in V\), \(\sum_{e^+=v} \log(A(e)) = 0\).
\end{theorem}
In~\cite{BMS23} we use this result to derive a formula for the derivative
with respect to \( \log \pi \) of certain iterated Vologodsky integrals.
\begin{theorem}[{\cite{BMS23}}]\label{deriter}
  Suppose \( \omega \) and \( \eta \) are two forms of the second kind on
  \( X \). Let
\[
  F= \int(\omega \int
\eta)
\] 
  with respect to the universal branch. The funcion
\[
 \frac{d}{d \log (\pi)} F
\]
 is constant on
  each domain \( U_v \)
  except near the singularities of \( \omega \) or \( \eta \), and
  associating this constant with the vertex \( v \) one gets a function on
  the dual graph which satisfies \[
   \Delta \frac{d}{d \log (\pi)} F_\omega (v) 
  = \frac{1}{2}\sum_{e^+=v} c_{\eta}(e) \res_e \omega- c_{\omega}(e)\res_e \eta -
  \sum_{e^+=v} \pair{\int \eta,\int \omega}_e \;.
  \] 
  Here, \( c_{\omega} \) and \( c_{\eta} \) are the harmonic cochains
  associated with \( \omega \) and \( \eta \) in Theorem~\ref{beszer} (with
  respect to the canonical branch) and the last term is the local index on an
  annulus (see~\cite{Bes98b}).
\end{theorem}

\section{$p$-adic heights}\label{ss-loch}
In this section we report on work with 
M\"uller and Srinivasan~\cite{BMS21}. In this work we describe a \( p \)-adic
analogue to the theory of heights associated to line bundles with adelic
metrics~\cite{Zhan95}. Our main interest is in applications to the problem of
finding rational point on curves through the method known as ``Quadratic
Chabauty''~\cite{BaBeMu12,Bal-Dog16,Bal-Dog17}, but the theory of heights we develop is of independent
interest. For analogy, let us first quickly review the classical theory.

Let \( F \) be a number field with a system of absolute values \( |~|_\nu
\) satisfying the product formula.
Let \( X \) be a smooth complete variety over \( F \).
\begin{definition}
  An adelic line bundle on \( X \) is a line bundle \( \L \) over \( X \)
  together with a collection of norms, \( \|~\|_\nu \) on the line bundle
  \( \L \otimes F_\nu \), for each place,
  finite or infinite, \( \nu \) of \( F \). This collection needs to
  satisfy certain conditions, which guarantee, among other things,
  that the associated height \[
    h_{\hat{L}}: X(F)\to \mathbb{R}, \quad h_{\hat{L}}(x)= \sum_\nu
    \log \|w\|_\nu
  \] 
  is a finite sum, where \( w \) is any vector in the fiber of \( L \) over
  \( x \).
\end{definition}
From the product formula it is immediate to see that the height indeed
depends only on \( x \) and not on the particular choice of \( w \). We
avoid discussing conditions on the norms for ease of presentation, since we
will mostly be concerned with the \( p \)-adic theory, and since canonical
norms that we want to study satisfy these automatically.

Following Tate's definition of canonical heights for Abelian varieties, one can
attempt to obtain canonical heights in various situations. The following
scenario, which we dub ``the dynamical situation'', is considered by
Zhang~\cite{Zhan95}(see also~\cite{Cal-Sil93}).
Suppose that $f\colon X\to X$ is an endomorphism $X$ and the line bundle \(
\L \) is such that there exists an integer \( d>1 \) and an isomorphism
\begin{equation}\label{eq:dyn-isom}
\beta\colon \L^{\otimes d
}\to f^\ast \L\;.
\end{equation}
Zhang constructs, at each place \( \nu \), a canonical norm on $\L_v$ for which
$\beta$ is an isometry (we will say that the metric is \emph{compatible}
with the dynamical situation), by starting with
any norm $\|\cdot\|_\nu$ and repeatedly replacing $\|\cdot\|_\nu$ with $(\beta^{-1}f^\ast
\|\cdot\|_\nu)^{\frac{1}{d}}$ and taking a limit. The resulting local
metrics turn out to give an adelic metric. In the case where \( X \) is an abelian variety, \( f:X \to X \) is the
multiplication by \( 2 \) map and \( \L \) is a symmetric line bundle on \(
X\), one gets the Tate canonical height associated with \( \L \).  

The \( p \)-adic analogue of the above picture is considered
in~\cite{BMS21}. The \( p \)-adic analogue of the system of absolute values
satisfying the product formula is a continuous id\`ele class character
 \begin{equation*}
   \chi = \sum_\nu \chi_\nu \colon  \mathbb{A}_F^\times/F^\times \to
   \Q_p\;.
 \end{equation*}
The local components \( \chi_\nu \) are the analogues of the logs of the absolute
values and the fact that the character vanishes on \( F^\times  \) is the
analogue of the product formula. Note the important difference from the
classical case: The primes above \( \infty \) play no role. 

For $\nu\mid p$, there is a decomposition
  \begin{equation}\label{tdefnd}
    \xymatrix{
      {F_\nu^\times}  \ar[rr]^{\chi_{\nu}} \ar[dr]^{\log_\nu} & &   \Q_p.\\
      & K_\nu\ar[ur]^{t_\nu}
  }
  \end{equation} 
for some $\Q_p$-linear trace map $t_\nu$ and some choice of a branch \(
\log_\nu \) of the \( p \)-adic logarithm. 
In analogy with the classical case we can define local norms, or rather,
logs of local norms, on a line bundle \( L \), to be functions \( \ell_\nu \) on the total
space of \( L \) minus the zero section, which satisfy the relation \[
 \ell_\nu(\alpha w) = \chi_\nu(\alpha)+\ell_\nu(w) \;.
\] 
However, at all places \( \nu \) not above \( p \) one has the property \(
\chi_\nu(\alpha)= \nu(\alpha) \chi(\pi_\nu)\), where \( \pi_\nu \) is a
uniformizer at \( \nu \) and we write \( \nu \) also for the associated
valuation, so
that it makes sense to divide the function by \( \chi_\nu(\pi_\nu) \). With this in mind
we make the following.
\begin{definition}
  Let \( K \) be a finite extension of \( \mathbb{Q}_q \), for some finite
  prime \( q \). Let \( X \) be a variety over \( K \) and let \( \L \) be
  a line bundle over \( X \). Then, a \( p \)-adic metric on \( \L \)
  is:
\begin{itemize}
  \item For \( q=p \) a \emph{log function} 
    \begin{equation*}
      \log: \L^{\times }: \L - 0 \to \overline{K}
    \end{equation*} 
    A function satisfying $\log(\alpha w)=
    \log_\nu(\alpha)+\log(w)$, $\alpha\in K^{\times }$, $w\in \L_x$, where
    \( \log_\nu \) is the branch of the \( p \)-adic logarithm determined
    by~\eqref{tdefnd}. 
    \item For $q\ne p$  a \emph{valuation} 
    \begin{equation*}
      \val: \L^{\times } \to \Q
    \end{equation*} 
    satisfying $\val(\alpha w)=
    \nu(\alpha)+\val(w)$
    where $\nu$ is the valuation on $K^{\times }$, normalized so that \(
    \nu(\pi_\nu)=1 \).
\end{itemize}
Let \( F \) be a number field, \( X \) a variety over \( F \) and \( \L \)
a line bundle over \( X \). An adelic metric on \( \L \) is a collection of
metrics for the completions of \( \L \) over all finite places of \( F \),
  satisfying a certain technical condition (that the valuation is almost
  always a ``model valuation'').
A line bundle \( \L \) over \( X \) together with an adelic metric will be
called an adelically metrized line bundle, usually denoted \( \hat{\L} \). 
 The $p$-adic height $h_{\hat{\L}}: X(F) \to \Q_p$ associated to the
  adelically metrized line bundle \( \hat{\L} \) is given by the formula \[
 h_{\hat{L}}(x)=
    \sum_{\nu | p} t_\nu \log_\nu(w)
    +\sum_{\nu \nmid p} \val_\nu(w) \cdot \chi_\nu(\pi_\nu)
  \] 
  where \( w \) is any vector in the fiber of \( \L \) over \( x \).
\end{definition}
As in the classical case, it is obvious that the definition of the height
is independent of the choice of \( w \).
\begin{remark}\label{gen-val}
  One can define valuations that take values in other groups. This will be
  useful later on (see Proposition~\ref{Volvaldef}). For the global theory and the construction of
  heights, only \( \mathbb{Q} \)-valued valuations are useful.
\end{remark}

The part of the theory which involves finite primes different from \( p \),
is in many respects joint between the \( p \)-adic and the classical case.
Indeed, given a valuation \( \val_\nu \), where \( \nu \) is a place above
a finite prime \( q\ne p \), such that the associated absolute value is \(
|~|_\nu = c^{\nu(~)} \) for some \( c\in (0,1) \), the function \[
  \|~\|_\nu = c^{\val_\nu(~)} 
\] 
is a norm in the classical sense. Clearly, it is not the most general type
of norm. However, it turns out that in many situations the norms one
considered are in fact of this form. Conversely, if we obtain a norm taking
values in \( c^{\Q}  \), then taking the log with respect to \( c \) gives
us a valuation.

Let us turn now to the situation above \( p \). As we mentioned for the
classical theory, defining heights using metrics can be done without any
serious restrictions on the type of metric used. However, in practice one
imposes some restrictions on the norms. In the \( p \)-adic situation, this
is both a restriction and a good way of constructing metrics. Recall that
if \( K \) is a finite extension of \( \mathbb{Q}_p \), and \( \L / X \) is
a line bundle over a smooth variety over \( K \), then a \( p \)-adic
metric on \( \L \) consists of a log function on the total space of \( \L
\) minus the zero section. It is natural to ask that this log function is a
Vologodsky function. We in fact ask for a far stronger condition: that log
is at most a second iterated Vologodsky integral. Let us call such log functions \emph{admissible} (note,
in~\cite{Bes00}
we called these log functions and what we now call log functions we called
there Pseudo log functions). The key result concerning these functions, allows
their characterization as well as construction using the notion of
curvature. This is described in the following Theorem.

\begin{theorem}[{\cite[Proposition~4.4]{Bes00}}]\label{log-curv}
Suppose \( X \) is proper and let \[
  \cup: \hdr^1(X / K)\otimes \Omega^1(X) \to \hdr^1(X /K)
  \] 
  be the cup product map. Suppose
 $ch_1(\L)$ is in the image of \( \cup \). Then one can associate to each
  admissible log function on \( \L \) a \emph{curvature form}, 
$\alpha\in \hdr^1(X / K)\otimes \Omega^1(X)$, such that
  $\cup(\alpha)=ch_1(\L)$ Conversely, any curvature form $\alpha$ cupping to 
$ch_1(\L)$ is the curvature of an admissible log function on $\L$, unique up to an
integral of $\omega\in \Omega^1(X)$.
\end{theorem}
We will not explain the construction of the curvature or the proof of the
theorem. An explicit description of a log function with a given curvature
is not so easy in general. However, for the case of
curves, a fairly easy description can be given as follows.
\begin{theorem}[{\cite[Section~3]{BMS21}}]\label{logdima}
  Let \( \L \) be a line bundle over a complete curve \( X \) over \( K \)
  and let \( \alpha \) be a curvature form cupping to \( ch_1(\L) \). Fix a
  rational section \( s \) of \( \L \). To determine a log function on \(
  \L \)  it suffices to determine \( \log (s)
  \). 
  Let $\alpha = \sum_i \omega_i \otimes[ \eta_i] \in  \Omega^1(X) \otimes \hdr^1(X)$ 
 for some holomorphic forms $\omega_i \in \Omega^1(C)$ and forms of the
  second kind $\eta_i$ with corresponding cohomology classes $[\eta_i] \in
  \hdr^1(C)$. Then, the formula
\begin{equation}\label{E:logfromitint} d \log_{\L}(s) = \sum \omega_i \int
\eta_i + \gamma,
\end{equation}
determines a log function on \( \L \) with curvature \( \alpha \). Here, 
$\gamma$ is any meromorphic form with the property that the right-hand
  side has only simple poles on \( X \) and its residue divisor is exactly
  \( \operatorname{div}(s) \)
\end{theorem}
We observe that the degree of freedom in choosing \( \gamma \) is exactly
adding a holomorphic form, so the degree of freedom in the log function is
exactly adding the integral of a holomorphic form, which is exactly the
degree of freedom in Theorem~\ref{log-curv}. 

If we try to copy the construction of canonical norms in Dynamical
situations to the \( p \)-adic setting, we immediately observe a problem:
The limiting procedure defining the norms will not converge \( p
\)-adically. Thus, another idea is required for finding canonical metrics.

For \( q\ne p \), one can hope that a canonical norm obtained through the
limiting procedure over \( \mathbb{R} \) actually takes values in some \(
c^{\mathbb{Q}}  \), hence a valuation may be obtained from it. This will
not work for \( q=p \), where the log function is truly \( p \)-adic.

Instead of using limiting methods, one may use the theory of the curvature,
Theorem~\ref{log-curv}. Consider then \( K \) a finite extension of \(
\mathbb{Q}_p \), a complete smooth variety \( X \) over \( K \), a line
bundle \( \L \) over \( X \), an endomorphism \( f:X\to X \) and an
isomorphism~\eqref{eq:dyn-isom}. If the log function \( \log_{\L} \) with
curvature \( \alpha \) makes \( \beta \) an isometry, then we have
\begin{equation}\label{eq:dyn-curve}
  f^\ast \alpha = d \alpha
\end{equation}
This is consistent with the fact that \( f^\ast ch_1(\L) = d ch_1(\L)\),
which needs to hold since \( \alpha \) cups to \( ch_1(\L) \). Conversely,
suppose that~\eqref{eq:dyn-curve} holds with \( \alpha \) a curvature form
cupping to \( ch_1(\L) \). Then, for any log function on \( \L \) with
curvature \( \alpha \), which exists by Theorem~\ref{log-curv}, The two log
functions, \( d \log_{\L} \) and \( \beta^{-1} f^\ast \log_{\L}  \), on \(
\L^{\otimes d} \), have the same curvature and therefore, again by
Theorem~\ref{log-curv}, defer by the integral of a holomorphic form on \( X
\). At this point one can study how this form is changed when \( \log_{\L}
\) is changed by the integral of a holomorphic form, and, depending on the
action of \( f^\ast \) on \( \Omega^1(X) \), it's often possible to make
\(\beta\) an isometry as required.

One case of the above was key to~\cite{BMS21}. There we construct canonical
metrics on a symmetric line bundle \( \L \) on an abelian variety \( X \).
In this case \( f \) is the multiplication by \( 2 \) map and \( d=4 \).
For further details see~\cite[4.2.1]{BMS21}.

At this point we apply the general theme of this work. Up to now we defined
the \( p \)-adic metric above \( p \) with a fixed choice of a branch,
dictated by the id\`ele class character \( \chi \).
Now, instead, we will consider the metric as a Vologodsky function with
respect to the universal branch of the logarithm, and differentiate with
respect to \( \log (p) \). As we will see, this gives valuations, in the
more general sense of Remark~\ref{gen-val}, i.e., with values in more
general groups than the additive group of the rationals.
\begin{proposition}\label{Volvaldef}
  Let \( \log_{\L} \) be a log function with the universal branch of the
  logarithm. Then \( d / d \log (\pi) \log_{\L} \) is a \( \overline{K}
  \)-valued valuation on \( \L \). We call it the \emph{Vologodsky
  valuation associated with the log function}.
\end{proposition}
\begin{proof}
  This is immediate from~\eqref{key-thing}.
\end{proof}
\begin{proposition}
 The Vologodsky valuation associated with a log function on \( \L \)
  depends only on the associated curvature \( \alpha \). Thus, we may
  denote it \( \val_\alpha \). This valuation is defined for every
  curvature form \( \alpha \) cupping to \( ch_1(\L) \).
\end{proposition}
\begin{proof}
  This is clear from Corollary~\ref{hol-indep}.
\end{proof}

The following is obvious.
\begin{proposition}
  Let \( \L \) be a line bundle on \( X \) and let \( \alpha \) be a
  curvature form cupping to \( ch_1(\L) \).
  If $g:Y\to X$ is a morphism, then $g^{\ast}\val_\alpha  =
  \val_{g^{\ast}\alpha}$
\end{proposition}
\begin{corollary}
  In the dynamical situation, suppose \( \alpha \) is a curvature form
  cupping to \( ch_1(\L) \) and satisfying \( f^\ast \alpha = d \alpha \).
  Then, \( \beta \) is an isometry (up to a constant) for the valuation \(
  \val_\alpha\).
\end{corollary}
\begin{proof}
The two valuations one obtains on \( \L^{\otimes d}  \) have the same
  curvature.
\end{proof}

Thus, in dynamical situation, there are (at least) two ways to construct a
compatible valuation locally at a place above the prime \( q \). 
One method is via limiting procedures using norms over \( \mathbb{R} \), hoping that the
resulting norm has values in some \( c^{\mathbb{Q}}  \). The other is via \( q \)-adic analysis and
curvature forms. The former method has several advantages
\begin{itemize}
\item It does not need to assume the existence of a well-behaved curvature
form
\item The latter method may not produce \( \mathbb{Q} \)-valued valuations.
The former method may produce norms which are not coming from valuations,
  but these may still be used in the theory of heights while non- \(
  \mathbb{Q} \)-valued valuations are useless there.
\end{itemize}

We note though that for any line bundle \( \L \) on an abelian
variety \( X \) the Vologodsky valuation associated with any curvature form
for \( \L \) is \( \mathbb{Q} \)-valued. In fact, it does not depend on the
curvature but only on the line bundle itself and is the canonical valuation
so may be used for the computation of the canonical height. This is proved
in~\cite[Theorem~9.24]{BMS21}.

The method of Vologodsky valuations has one very important advantage. If a
valuation is constructed in a dynamical situation for the line bundle \( \L
\) over \( X \) using the limiting procedure, then there would be some way to control, compute or
approximate it over \( X \). However, if we pull back via \( g:Y\to X \),
there will be no obvious way to compute the pulled back valuation over \( Y
\). On the other hand, using Vologodsky valuations, the pulled back
valuation is associated with the pulled back curvature, and is therefore,
at least in principle, computable. This is especially true if \( Y \) is a
curve, as Theorem~\ref{logdima} provides an explicit description of the log
function in terms of the curvature, and Theorem~\ref{deriter} provides a way of computing the
associated valuation. In~\cite{BMS23} we use this to explicitly compute the
pullback to a curve of the canonical valuation on a symmetric line bundle
on the Jacobian of the curve, with applications to Quadratic Chabauty.

\section{The unipotent Albanese map}\label{sec:unipalb}

The classical Albanese map between a variety and its Albanese variety has a
non-abelian analogue, the \emph{Unipotent Albanese map} first appearing in
its Hodge 
incarnation in the work of Hain~\cite{Hain87}
and later becoming a cornerstone of Kim's non-abelian Chabauty's
method~\cite{Kim05,Kim09}.

Our Goal is to discuss the case of curves with semi-stable reduction, but
doing so in a rather informal manner. In particular, we do not claim any
relation with the \'etale theory.

Let us discuss things first is some generality. Suppose \( X \) is a
variety over a field \( K \), which for the moment can be any field.
 We assume that we can define for \( x_0 \in
X(K) \) some sort of fundamental group \( G=\pi_1(X,x_0) \), possibly with
some extra structure, and, given
another rational point \( x\in X(K) \), a fundamental path space \(
\pi_1(X,x_0,x) \), again with some extra structure, which is a right torsor
for \( G \) with compatible structures (using the theory of
Tannakian categories one can often streamline the meaning of these
compatible structures).
\begin{definition}
  
Let \( \alb(G) \) be the set of isomorphism classes
  of \( G \)-torsors with compatible structure and define the
  \emph{abstract Albanese map} as the map
  \[
  \alpha_X: X(K) \to \alb(\pi_1(X,x_0))
\] 
sending \( x\in X(K) \) to the isomorphism class of
the torsor \( \pi_1(X,x_0,x) \).
\end{definition}

Perhaps the most obvious example of an abstract Albanese map is the one for
which \( \pi_1 \) is the \'etale fundamental group (respectively \'etale
path space) of \( \overline{X}
:= X \otimes_K \overline{K} \), which is a group (respectively a torsor) with an action of \(
\Gal(\overline{K} / K) \). In this case we clearly have \[
  \alb(\pi_1(X,x_0))=H^1(\Gal(\overline{K} / K),\pi_1(\overline{X},x_0))
\;,\]  
the non-abelian cohomology of \( \Gal(\overline{K} /K) \) with coefficients
in \( \pi_1(\overline{X}, x_0) \).

In the second example \( K \) is a finite extenion of \( \mathbb{Q}_p  \),
\( X \) is a proper variety over \( K \) with good reduction 
and the fundamental group is the unipotent de Rham one, the pro-algebraic
group associated by Tannakian duality to the category of unipotent
integrable connections on \( X \) with the fiber functor of taking the
fiber at \( x_0 \). This has a filtration and an action of a Frobenius
automorphism (we take the linear Frobenius and not the semi-linear one for
convenience), and path spaces have the same structure. In~\cite{Bes10a} we
proved the following result.
\begin{theorem}\label{Heidelberg}
  We have a canonical isomorphism
  \begin{equation}\label{eq:albdesc}
    \alb(\pi_{1}^{dR}(X,x_0) ) \isom  F^0\backslash \pi_1^{dR} (X,x_0)(K) 
  \end{equation} 
  Furthermore, there exists a correspondence between regular functions on
  \(  \pi_1^{dR} (X,x_0) \), 
  left
  invariant under \( F^0 \), and Coleman functions on \( X \) such that if the function \( f \)
  corresponds to the Coleman function \( F \) on \( X \), then \[
    f(\alpha_x(x))=F(x),\quad x\in X(K)\;.
  \] 
\end{theorem}
For the proof we need to recall that function on an affine group scheme \(
G\) over a field \( K \) can be given by ``matrix coefficients''. These are
defined by triple \( (V,v,s) \), where \( V \) is a finite dimensional
representation of \( G \), \( v\in V \) and \( s\in V^\ast \). Such a
triple defines a function on \( G \) by the formula \[
  (V,v,s)(g)= s(gv)\;.
\]
There are obvious relations between such triple that produce the same
function. In particular, if \( T:V\to V' \) is a \( G \)-map, then the
triple \( (V,v,T^\ast s') \) and $(V', Tv, s')$ produce the same
function. If \( G \) is the fundamental group of a neutral Tannakian
category \( \mathbb{T} \), with fiber functor \( \omega \), then these
triple are given in terms of \( \mathbb{T}, \omega \) as \( (V,v,s) \)
where \( V\in \mathbb{T} \), \( v\in \omega(V) \) and \( s\in
\omega(V)^\ast \). This allows a direct construction of the underlying Hopf
algebra. In the standard definition one uses in the construction \(
\omega(V) \otimes \omega(V)^\ast \) but it is easy to see that simple tensors
suffice.
\begin{lemma}\label{mylemma}
  Let \( \delta: H \to G \) be a homomorphism of affine group schemes over
  \( K \). A function \( f:G \to K \) is \( \delta(H) \) left invariant if
  and only if it can be represented by a triple \( (V,v,s) \) where \( s
  \) is an \( H \)-homomorphism \( V\to K \), with \( K \) having the
  trivial \( H \)-action.
\end{lemma}
\begin{proof}
Clearly if such a representation exists, the function is left invariant. Conversely,
  suppose \( f \) is \( H \) invariant from the left and is represented by a
  triple \( (V,v,s) \). By replacing \( V \) with a sub \( G \)-module we
  may assume that \( v \) generates \( V \) as a \( G \)-module. But then
  the equation $s(\delta(h) g v) = s(gv)$ for every \( g\in G \) and \(
  h\in H \) implies that \( s (\delta(h)w) = s(w) \) for every \( w\in V
  \).
\end{proof}
\begin{proof}[Proof of Theorem~\ref{Heidelberg}]
  By~\cite[Corollary~3.2]{Bes99} each torsor \( P \) for \( G=\pi_1^{dR}(X,x_0)  \) has a canonical Frobenius invariant element \(
  \gamma_\phi \) (In Section~\ref{sec:colvol} we called it \( \gamma_{x,y} \), a path connecting
  the two points, but the proof of the theorem, which is an immediate
  consequence of~\cite[Theorem~3.1]{Bes99} clearly applies to any path space for
  \( G \)). Pick any \( \gamma_{dR} \in F^0 P \). Then \(
  \gamma_{\phi}= \gamma_{dR}g \) for a unique \( g\in G \). Any other such
  choices of \( \gamma_{dR} \) will be of the form \( \gamma_{dR} h \) for
  some \( h\in F^0 G\), replacing \( g \) by \( hg \). This establishes the
  bijection~\eqref{eq:albdesc}. Suppose now that
  \( f:  G \to K \) is a
  regular function that is invariant from the left under \( F^0  \).
  By~\cite[Remark~3.8]{Had11} the subgroup \( F^0 G \) is
  the image of \( H \), the fundamental group of the neutral Tannakian
  category of unipotent vector bundles, under the
  map induced by the functor that forgets the connection on bundles with
  connections.
 Let \( ((M,\nabla ),v,s) \) be a
  representation for \( f \) whose existence is guaranteed by
  Lemma~\ref{mylemma}, so \( s \) is an \( H \)-homomorphism. This
  immediately translates to it being the fiber at \( x_{0} \) of a map of
  vector bundles, which we continue to call \( s \). With the new
  interpretation of \( s \), the triple \( ((M,\nabla ),v,s) \) is now a
  Coleman function \( F \). Its value at \( x \) is by definition \[
    F(x) = s(\gamma_\phi v)\;.
  \] 
  The value of \( f \) on \( \alpha(x)= g  \) is \( s(gx) \). The result
  now follows because \[
    s (\gamma_\phi v) = s (\gamma_{dR}(gv)) = \gamma_{dR} (s (g v) = s (g
    v)\;.
  \]
  The last two equalities follow because \( \gamma_{dR} \), being in \( F^0
  \), commutes with arbitrary morphisms of vector bundles, in particular \(
  s\), and on the trivial line bundle it is just the identity.
\end{proof}
Next we consider a variety \( X \) with semi-stable reduction over \( K \). 
We then \emph{should} have a similar picture with a fundamental group
$G=\pi_1(X,x_0)$ in the category of filtered
$(\varphi,N)$-modules, and a corresponding
Albanese map. To be more precise, There is a triple \( G=(G_{cr},G_{dR},I_\pi) \)
consisting of an affine group scheme \( G_{cr} \) over \( K_0
\) with Frobenius and monodromy operators, an affine group scheme
\( G_{dR} \) with a filtration, and a comparison isomorphism
\begin{equation*}
I_{\textup{HK},\pi}: G_{cr}\otimes_{K_0} K \to G_{dR}
\end{equation*}
depending on the choice of a uniformizer \( \pi \). The dependency on \(
\pi \) is such that if we change to \( \pi' \) we have, exactly as
in~\eqref{chofuni},
\begin{equation}\label{gpdep}
  I_{\pi'}= I_\pi \circ \exp (\log (\pi' / \pi )N) \;.
\end{equation}
Note that \( N \) (and any multiple of it), being a topologically nilpotent derivation of a Hopf
algebra, can be
exponentiated to an automorphism of the algebra, hence of the associated
affine group scheme.
\begin{remark}
  We intentionally wrote \emph{should have}. There have been many
  approaches to the unipotent fundamental group in the log scheme or
  semi-stable case (a probably not inclusive list
  includes~\cite{KimHain04,KimHain05,Kat-Lit21,Vol01,Deg-Niz18,Shi00}
). Not all of them state all required results. For
  example, to the best of our knowledge, the dependency on the uniformizer
  of the Hyodo-Kato isomorphism is only hinted to in the introduction
  to~\cite{KimHain05}. Also, while it is often claimed that all of these
  approaches are the same, this seems to never be actually proved,
  excluding some results in~\cite{CDS23}. In
  particular, to actually make all of this into actual theorems, the theory
  should be compared with the theory of Vologodsky.
\end{remark}
The Albanese set \( \alb(G) \) is the set of isomorphism classes of \( G \)
torsors with compatible structures \( P=(P_{cr},P_{dR}, I_\pi) \) in the
obvious sense. 
\begin{proposition}\label{prodstruc}
  There exists a bijection
  \[
    P\mapsto (P_{c,\pi},P_d): \alb(G) \to  (F^0\backslash G_{dR}) \times \nill(G_{cr})
  \]
where \( \nill(G_{cr})
\) classifies compatible \( N \) on the trivial torsor \( G_{cr} \)
  with its Frobenius.
\end{proposition}
\begin{proof}
Let \( P \) be a \( G \) torsor. By~\cite{Vol01} the torsor \( P_{cr} \)
  has a canonical path \( \gamma_\phi \). This path is fixed by Frobenius. As a torsor with a
  Frobenius action, \( P_{cr} \) is therefore trivial and \( P_{dR} \) is
  trivialized by \( I_\pi(\gamma_\phi) \). The set \( \alb(G) \) is
  thus the product of the possible Filtrations on \( G_{dR} \) as a \(
  G_{dR} \)-torsor, with \( \nill(G_{cr}) \). The former is clearly
  classified by \(F^0 \backslash G_{dR}\). 
\end{proof}
Note that since the Vologodsky path is not killed by monodromy, there is
still some wiggling room for the monodromy, and \( \nill(G_{cr}) \) is not trivial.
Also note that the map to \(F^0 \backslash G_{dR}  \) depends on the choice of the
uniformizer \( \pi \), since the Vologodsky path needs to be imported via
\( I_\pi \) to \( G_{dR} \) in order to define it.

The map in Proposition~\ref{prodstruc} is explicitly given as follows
\begin{equation}\label{prodexpl}
    P_{c,\pi} = \gamma_{dR}^{-1} I_\pi(\gamma_\phi)
  \;,\quad P_d = \gamma_\phi^{-1} N(\gamma_\phi)\;.
\end{equation}
The element \( N(\gamma_\phi) \) belongs to \( \operatorname{Lie}(G_{cr})
\) and characterises \( P_d \). Vologodsky theory puts some restrictions on
this element, which we will not attempt to describe here.
\begin{remark}
  The product decomposition of the Albanese was already observed by Betts.
\end{remark}
\begin{definition}
  The compositions of the Albanese map \( X(K)\to \alb(G) \)
  with the projections on the factors \( F^0 \backslash G_{dR}  \) and \( \nill(G_{cr})  \) will
  be called the continuous and discrete components of the semi-stable
  Albanese map respectively, denoted \( \alpha_{c,\pi} \) and \( \alpha_d \)
  respectively.
\end{definition}
\begin{theorem}
For each choice of uniformizer \( \pi \), the continuous component of the
  semi-stable Albanese map is computed using Vologodsky integrals:   
  There exists a correspondence between regular functions on
  \(  G_{dR}\), 
  left
  invariant under \( F^0 \),  and Vologodsky functions with respect to the branch of
  the logarithm determined by \( \log (\pi)=0 \), such that if the regular function \( f: G_{dR}\to K \)
  corresponds to the Vologodsky function \( F \) on \( X \), then \[
    f(\alpha_{c,\pi}(x))=F(x),\quad x\in X(K)\;.
  \] 
  The discrete component of the semi-stable Albanese is the logarithmic derivative, with
  respect to \( \log (\pi) \) of the continuous component: \[
    P_{c,\pi}^{-1} \frac{d}{d \log (\pi)} P_{c,\pi} = I_\pi(P_d)\;.
  \]  
\end{theorem}
\begin{proof}
The statement about the continuous component is word for word the proof in
the good reduction case. The second statement is obtained by
  deriving~\eqref{prodexpl}
  with respect to \( \log (\pi) \) using~\eqref{gpdep}.
  result
\end{proof}
There similarity between the construction in this section and the
discussion of syntomic regulators in Section~\ref{sec:syntomic}, and in particuar
Theorem~\ref{synderi}, should be obvious. We have no attempted though to
spell out what this result means in specific situations.

\section{Toric regulators}\label{sec:toric}
In this section we recall the main results of our paper~\cite{Bes-Ras17} with
Wayne Raskind and relate it to the main theme of this paper.

In~\cite{Bes-Ras17} we defined (conditionally since some reasonable conjectures
are assumed) a new type of regulator maps.
It applied to varieties $X$ over a \( p \)-adic field \( K \) with so
called ``totally degenerate'' reduction. This means that \( X \) is smooth
projective, has strictly semi-stable reduction, and the components of its
reduction have all of their cohomologies generated by cycles
(see~\cite[Section~2]{Ras-Xar07} for a precise definition).
The target of the toric regulator is a ``toric intermediated Jacobian'' which
is a $p$-adic torus $\coker: T^0 \to T^{-1} \otimes
K^\times$, where $T^\ast$ are finitely generated $\Z$-modules whose
definition we recall shortly. Note that \( T^0 \) and \( T^{-1}  \) may not
be of the same rank.

Some simple examples of toric regulators include:
\begin{itemize}
    \item 
The identity map $H_{\mathcal{M}}^1(K,1) = K^\times \to K^\times$
    \item 
For $E/K$ a Tate elliptic curve $\mathbb{G}_m /q^{\Z}$ the identity map
$$H_{\mathcal{M}}^2(E,1)_0 \isom E(K) \to K^\times / q^{\Z}$$
    \item 
More generally, for a Mumford curve $X$ with Jacobian $J$, the Abel map
$X(K) \to J(K)$, where 
$J(K)$ is viewed as a \( p \)-adic torus via its $p$-adic
  uniformization~\cite{ManDri73}.
    \item 
  The rigid analytic regulator of P\'al~\cite{Pal10}.  For $X$ a Mumford curve with dual graph $\Gamma$, it is a map
$$ \reg_P: K_2(X)\to \hh(\Gamma,K^\times) $$
the right-hand side being harmonic cochains with values in \( K^\times  \).
\end{itemize}

We now recall the construction of the toric regulator.
The special fiber $Y$ decomposes as a union $Y= \cup_{i=1}^n Y_i$.
For a subset
$I\subset \{1,\ldots,n\} $ we let 
$$ Y_I = \bigcap_{i\in I} Y_i\;.$$
Fix integers
$k,r \ge 0$. For each finite prime \( \ell \), let 
\[
M_\ell=H_{\et}^{k}(X\otimes_K \bar{K}, \Z_\ell)\;.
\]
\begin{theorem}[{\cite{Ras-Xar07}}]\label{weight-fil}
Ignoring finite torsion and cotorsion
  there are finitely generated $\Z$-moduels $T^i$, $i\in \Z$ (depending on \( k \) and \( r \)) such that for
each $\ell$ the Galois module $M_\ell(r)$ has an increasing filtration $U^\bullet$ with
$\gr_U M_\ell(r) = \oplus_i T^i \otimes \Z_\ell(-i)$.
\end{theorem}
\begin{proof}
  See~\cite[Corollaries 1 and 6]{Ras-Xar07}. The proof uses the
  Rapoport-Zink weight spectral sequence~\cite{RapZin82} for $\ell\ne p$
and result by
  Mokrane, Hyodo and Tsuji~\cite{Mok93,Hyo88,Tsu99} when $\ell=p$.
\end{proof}

Recall that the residue field of \( K \) is denoted \( \rf \). Put
\begin{itemize}
\item  
$\bar{Y}_I:= Y_I\otimes \bar{\rf}$
\item
$\Ym = \bigcup_{|I| =m} \bar{Y}_I$
\item
$\cijk= \Cijk$, $k\ge \max(0,i)$
\item
$  \cij = \bigoplus_k \cijk $
\end{itemize}
For every index \( s \)
let
$I_s = I-\{i_s\}$
and consider the obvious inclusion
$\rho_s: Y_I \to Y_{I_s} $. Let
\begin{align*}
  \theta_{i,m} &= \sum_{s=1}^{m+1} (-1)^{s-1} \rho_s^\ast: CH^i(\Ym) \to CH^i(\Ymm{m+1})\;,\\
  \delta_{i,m} &= \sum_{s=1}^{m+1} (-1)^{s} \rho_{s\ast}:  CH^i(\Ymm{m+1}) \to CH^{i+1}(\Ym)\;,\\
  d' &= \bigoplus_{k\ge \max(0,i)} \theta_{i+j-k,2k-i+1}\;, \\
  d'' &= \bigoplus_{k\ge \max(0,i)} \theta_{i+j-k,2k-i}\;,
\end{align*}
Finally, let
\[
 d_j^i = d'+d'': \cij \to C_j^{i+1} 
\;,\] 
set
\[
T_j^i := H^i(C_j^\bullet)
\;,\] 
and
renumber for ease of notation:
\[ 
T^i = T_{i+r}^{k-2r-2i}\;.
\] 
The monodromy map:
\begin{equation}\label{monmap}
N :T^i \to T^{i-1}
\end{equation}
given by ``identity on identical
components''.

There is, in general, a non-trivial action of \( \Gal(\overline{K} / K) \)
on the \( T \)'s. However, this will factor through a finite quotient and
so, for simplicity, we will assume that this action is in fact trivial,
which may be achieved after a finite field extension.

The quotient $M'_\ell = U^0 M_\ell(r) / U^{-2} M_\ell(r)$
sits in a short exact sequence
\[
 0\to T^{-1} \otimes \Zl(1) \to M_\ell^\prime \to  T^0 \otimes \Zl \to 0 \;.
\] 
Recall the
Bloch-Kato group of geometric cohomology classes $H_g^1$~\cite[Section~3]{Blo-Kat90} which in
particular satisfies \[
  H_g^1(K,\Zl)\isom 
 K^{\times(\ell)}
\]
where the right-hand side is the \( \ell \)-completion of the
multiplicative group of \( K \). It has a completed valuation map
\[
  \val: 
 K^{\times(\ell)} \to \Zl\;.
\] We have a boundary map
\[
\ntl: T^0 \otimes \Zl \to H_g^1(K,T^{-1}\otimes \Zl(1) )\isom T^{-1}
\otimes K^{\times(\ell)}\;.
\] 

\begin{lemma}[{\cite[Corollary~1.7]{Bes-Ras17}}]
The composition $T^0 \otimes \Zl \xrightarrow{\ntl} T^{-1} \otimes K^{\times(\ell)}
  \xrightarrow{\val} T^{-1} \otimes \Zl$ equals $N\otimes \Zl$, where \( N
  \) is the monodromy map~\eqref{monmap}.
\end{lemma}
\begin{corollary}{{\cite[Corollary~1.9]{Bes-Ras17}}}
There exists a homomorphism $T^0  \xrightarrow{\nt} T^{-1}\otimes K^{\times } $
  such that $\ntl = \nt
\otimes \Zl$ for each $\ell$.
\end{corollary}
\begin{definition}{{\cite[Definition~1.10]{Bes-Ras17}}}
We define the toric intermediate Jacobian
$$\hT(X,k+1,r) := \coker \nt . $$
\end{definition}
Here, and it what follows, observe a shift of indexing \( k\to k+1 \) which
matches~\cite{Bes-Ras17} but defers from Section~\ref{sec:syntomic}.

We next define the toric regulator completed at $\ell$.
The \'etale regulator map,
\[
 \reg_\ell: \hmot(X,k+1,r)_0 \to
H_g^1(K,M_\ell(r)) 
\] 
(see Remark~\ref{toric-remark}),
factors, up to possibly multiplying by some integer to avoid finite torsion
and cotorsion issues, via $H_g^1(K,U^0 M_\ell(r)) $ (recall that \( U^{\bullet}\) is the
filtration defined in Theorem~\ref{weight-fil}) because
$H^0(K,\Ql(j))= H_g^1(K,\Ql(j))=0$ for $j<0$~\cite[Example~3.9]{Blo-Kat90}.
Applying the projection $U^0 \to M'_\ell$ we get, again after possibly
multiplying by some integer, 
$$ \reg_\ell^{\prime\prime}: \hmot(X,k+1,r)_0 \to H_g^1(K,M_\ell^\prime) \isom
\coker( T^0 \otimes \Zl \xrightarrow{\ntl} T^{-1} \otimes K^{\times(\ell)})
.
$$

The idea is now to construct the toric regulator in such a way that the
above regulator at the prime \( \ell \) is its \( \ell \)-completion. For
this, however, we need to recall the theory of the
Sreekantan regulator~\cite{Sre10}. Sreekantan's motivation was to obtain
analogues of Beilinson's conjectures in the function field case. He defines a 
``Deligne style cohomology'',
$ H_{\mathcal{D}}^{k+1}(X,\Q(r))$.
We do not recall the general definition
which involves similar ingredients to our definition of the lattices \( T^i
\). In fact, assuming completely degenerate reduction we have, for most
ranges of indices, an isomorphism~\cite[(2.2)]{Bes-Ras17},
\[
 H_{\mathcal{D}}^{k+1}(X,\Q(r))\isom \coker(T_{\Q}^0  \to T_{\Q}^{-1}) \;.
\] 
Sreekantan defines his Deligne cohomology in the relevant range as
higher Chow groups of the special fiber 
\begin{equation*}
  H_{\mathcal{D}}^{k+1}(X,\Q(r)) \isom CH^{r-1}(Y,2r-k-2)\otimes \Q\;.
\end{equation*}  
Assuming standard conjectures this is equivalent to a definition using the
Consani complex~\cite[Proposition~2.3]{Bes-Ras17}.
Sreekantan defines the ``regulator'',
  \begin{equation*}
    r_{\mathcal{D}}: \hmot(X,k+1,r) \to H_{\mathcal{D}}^{k+1}(X,\Q(r))\;,
  \end{equation*}
as the boundary map in higher Chow
groups.
  \begin{equation*}
    \hmot(X,k+1,r) \isom CH^r(X,2r-k-1)\xrightarrow{\partial}  CH^{r-1}(Y,2r-k-2)\otimes \Q\;.
  \end{equation*}
In~\cite{Bes-Ras17} we made the following conjectures.
\begin{conjecture}{{\cite[Conjecture~1]{Bes-Ras17}}}\label{theconj}
For each prime $\ell$ the valuation of the toric regulator at $\ell$ is the Sreekantan regulator tensored with $\Ql$.
\end{conjecture}
We show that this conjecture suffices to prove the existence of the toric
regulator.
\begin{theorem}{{\cite[Theorem~1.14]{Bes-Ras17}}}
  Assuming Conjecture~\ref{theconj}, There exists a map, called the
  \emph{toric regulator},
  \begin{equation}\label{eq:treg}
  \hmot(X,k+1,r) \xrightarrow{\regt}  \hT(X,k+1,r)\;,
  \end{equation}
giving the toric regulator at $\ell$ for each $\ell$ by completion.
\end{theorem}

The toric regulator, when it exists, is related with the syntomic regulator
of Section~\ref{sec:syntomic}. The relation is roughly that
the syntomic regulator is the logarithm of the toric regulator. One reason
to expect such a relation, is the case of zero cycles. In this case, the
toric regulator is just the Albanese map into an appropriate \( p \)-adic
uniformization of the Albanese, while the syntomic regulator is the
Albanese map composed with the logarithm of the Albanese variety, which is
compatible with the logarithm on the torus uniformizing the Albanese.
In general, we have the following.
\begin{theorem}[{\cite[Theorem~4.6]{Bes-Ras17}}]\label{torislogsyn}
  Let \( X \) be a variety with totally degenerate reduction over \( K \),
  \( k,r \) positive integers,
  and let \( T^{-1}  \) and \( T^0 \) be as in Theorem~\ref{weight-fil}. Let
$V=H_{\et}^{k}(X\otimes_K \bar{K}, \Qp(r))$. Then the toric regulator at \(
  p\) exists, there is a canonical map
  \[
      \hst^1(K,V) \to 
      \DR(V)/(F^0+ T^0\otimes \Qp)
  \]  and we have a
commuting diagram
  \begin{equation*}
    \xymatrix{
      {\hmot(X,k+1,r)_0} \ar[r]^{{\rsyn}} \ar[d] & \hst^1(K,V) \ar[r] &
      \DR(V)/(F^0+ T^0\otimes \Qp) \ar[d] \\
      {\hT(X,k+1,r)} \ar[rr]^{\log} && T^{-1}\otimes K / T^0\otimes \Qp .
    }
  \end{equation*}
\end{theorem}

An obvious expectation from the result above is that the discrete part of
the syntomic regulator, which is the derivative with respect to \( \log
(\pi) \) of the continuous part, should be related to the Sreekantan
regulator. Let us briefly indicate how this can be seen in the case of \(
K_2 \) of curves. As mentioned after Theorem~\ref{k2sstwo},  the expression for the cup product of
the  (continuous part of the) regulator with a holomorphic form \( \omega
\) given there is a special case of a
more general expression valid for any form of the second kind.
The expression we need is found at Proposition~3.7 in~\cite{Bes18} which is the same as the expression for
  the regulator by the main
  theorem~\cite[Theorem~1.2]{Bes18}. This expression is
  \begin{equation*}
    \sum_v \pair{\log(f),F_\omega;\log(g)}_{{(U_v-Z)^\dagger}} - \sum_e c_\omega(e)\cdot \pair{\log(f),\log(g)}_e\;.
  \end{equation*}
  Here, $Z$ is a subset containing all the singularities of $f$ and
  $g$. The first term in this expression involves the triple index, which
  we do not explain here. In~\cite{Bes-Ras17} we pointed out that in the
  totally degenerate reduction case,  when all the components of the reduction are
  projective lines, this term vanishes because all those triple
  indices vanish
  by~\cite[Proposition~8.4]{Bes-deJ02}. This led to a formula for the
  regulator, which we proved is the logarithm of the P\'al regulator,
  consistent with Theorem~\ref{torislogsyn}. Without assuming completely degenerate reduction, it
  is still true that this first term is independent of the choice of
  branch. Recall that the formula applies to cup products with forms which are in the kernel of \( N \). By~\eqref{cdeplog} we see that for such forms \( c_\omega(e) \) is independent of the branch. By Theorem~\ref{synderi} we see that \[
    \rsyn^d(\{f,g\}) \cup \omega = 
   \sum_e
    c_\omega(e)\cdot \frac{d}{d \log (\pi)}\pair{\log(f),\log(g)}_e\;.
  \] 
 By Definition~\ref{tworegdef}, \( \rsyn^d \) lands in \( D^{-2}  \), which, in our case, is the
 weight \( 2 \) subspace of \( \hdr^1(X / K) \) (more precisely, it is a \( K_0
 \) subspace). This subspace is mapped isomorphically by \( N \) on the
 weight \( 0 \) space, which is isomorphic to $\hh(\Gamma,K)$. In addition, it is
 also dual to the weight \( 0 \) subspace via the cup product pairing.
 By~\cite[Corollary~5.3]{Bes17} the pairing \[
   \hh(\Gamma,K)\times \hh(\Gamma,K) \to K\;,\quad (x,y)\mapsto x\cup N^{-1}y\;, 
 \] 
 is none other than the pointwise scalar product \( \sum_e x(e)\cdot y(e)
 \). Thus, assuming that the main result of~\cite{Bes18} holds without
 restrictions on \( \omega \), or that the cochains \( c_\omega \) suffice
 to cover all of \( \hh(\Gamma,K) \), as is the case for example for
 Mumford curves, we get the conjectural formula \[
  N \rsyn^d(\{f,g\}) = \left(e\mapsto 
    \frac{d}{d \log (\pi)}\pair{\log(f),\log(g)}_e
  \right)\;.
 \] 
Let us see how this matches with the conjecture that \( \rsyn^d \) factors
via 
the Sreekantan regulator. This map
\begin{equation*}
  K_2(X) \to \hh(\Gamma,\Z)
\end{equation*} 
is derived from the boundary map in
K-theory of the integral model \( \mathcal{X} \), which in this case is
given by tame symbols at the irreducible components of the reduction,
\begin{equation*}
 \{f,g\}  \mapsto (v\mapsto h_v = t_{T_v}(f,g))
\end{equation*} 
followed by the map that send a collection
\[
(v\mapsto h_v\in\kappa(T_v)^{\times
})\text{ such that } \sum_v \operatorname{div}(h_v)=0 \text{ on } \mathcal{X}
\]
to
\begin{equation*}
 e \mapsto \ord_e
 \left( h_{e^+} \right)\;,
\end{equation*}
where here \( e \) is viewed as the intersection point of the components \(
T_{e^+}\) and \( T_{e^-} \) and the condition implies that \( \ord_{e}
h_{e^+} = - \ord_e h_{e^-} \).
To see how the Sreekantan regulator is related with \( \rsyn^d \) it
suffices to prove the following.
\begin{lemma}
  For any two rational functions \( f,g \) on \( X \) and any annulus \( e \) we have
 \begin{equation*}
    \frac{d}{d \log (\pi)}\pair{\log(f),\log(g)}_e
   = \ord_e
 \left( h_{e^+} \right)\;.
 \end{equation*} 
\end{lemma}
\begin{proof}
It is enough to consider the case when
$ \ord_{T_v}(f) = 1$ while 
  $ \ord_{T_v}(g) = 0$, with \( v=e^+ \).
  Then locally on \( U_v \) near \( e \)
  we can write \( f=\pi f_0 \) and
  \begin{align*}
    \pair{\log (f),\log (g)}_e &=
    \pair{\log (f_0),\log (g)}_e + \pair{\log (\pi),\log (g)} \\
    &\phantom{=} \pair{\log (f_0),\log (g)}_e + \log (\pi)\res_e d \log (g)\;,
  \end{align*} so that
\begin{equation*}
  \frac{d}{d\log(p)}\pair{ \log(f) ,\log(g) }_e= \res_e d\log(g)=
  \ord_e(g|_{T_v})
\end{equation*} 
where, for this last equation it suffices to consider \( g \) which is a
  local equation for \( T_{e^-} \) and use~\cite[Lemma~2.1]{Bes-Zer13}. 
\end{proof}

\providecommand{\bysame}{\leavevmode\hbox to3em{\hrulefill}\thinspace}
\providecommand{\MR}{\relax\ifhmode\unskip\space\fi MR }
\providecommand{\MRhref}[2]{%
  \href{http://www.ams.org/mathscinet-getitem?mr=#1}{#2}
}
\providecommand{\href}[2]{#2}

\end{document}